\documentclass[a4paper, 12pt]{article}
\usepackage{geometry}                
\usepackage{amssymb,amsmath,scalerel}
\usepackage[british]{babel}
\usepackage{amsthm}
\usepackage{thmtools, thm-restate}
\usepackage[plainpages = false]{hyperref}
\usepackage[capitalize,noabbrev]{cleveref}
\usepackage[utf8]{inputenc}
\usepackage{enumerate}
\usepackage{xcolor}
\usepackage{tikz-cd}
\usetikzlibrary{calc}
\usepackage[backend = biber, style = numeric, giveninits, maxbibnames=99]{biblatex}
\usepackage{tabularx}

\DeclareMathOperator{\Aut}{Aut}
\DeclareMathOperator{\ch}{ch}
\DeclareMathOperator{\Char}{char}
\DeclareMathOperator{\End}{End}
\DeclareMathOperator{\ev}{ev}
\DeclareMathOperator{\GL}{GL}
\DeclareMathOperator{\Id}{Id}

\DeclareMathOperator{\Par}{Par}

\DeclareMathOperator{\Span}{span}
\DeclareMathOperator{\SpecR}{Spec_{R}}
\DeclareMathOperator{\supp}{supp}
\newcommand*{\B}{\mathcal{B}}
\newcommand*{\bb}{\mathfrak{b}}
\newcommand*{\C}{\mathbb{C}}
\newcommand*{\CC}{\mathcal{C}}
\newcommand*{\CCinf}[1]{\CC_{#1}^{\infty}}
\newcommand*{\cf}{cf.\;}
\newcommand*{\Cf}{Cf.\;}
\newcommand*{\eg}{e.g.\;}

\newcommand*{\ff}{\mathfrak{f}}
\newcommand*{\ffsupport}[1]{\pp_{#1}}
\newcommand*{\fflength}[2]{\bb_{#1, #2}}
\newcommand*{\fgtf}{finitely generated torsion-free }
\newcommand*{\g}{\mathfrak{g}}
\newcommand*{\HH}{\mathcal{H}}
\newcommand*{\HHinf}[1]{\HH_{#1}^{\infty}}
\newcommand*{\ie}{i.e.\;}
\newcommand*{\inftynorm}[1]{|{#1}|_{\infty}}
\newcommand*{\inftynormb}[1]{\left|{#1}\right|_{\infty}}
\newcommand*{\inv}[1]{{#1}^{-1}}

\newcommand*{\M}{\mathcal{M}}

\newcommand*{\m}{\mathfrak{m}}

\newcommand*{\mrc}{\m_{r, c}}
\newcommand*{\N}{\mathbb{N}}
\newcommand*{\nn}{\mathfrak{n}}
\newcommand*{\nrc}{\nn_{r, c}}
\newcommand*{\polZ}[1]{\Z[x_{1}, \ldots, x_{#1}]}
\newcommand*{\powZ}[1]{\Z[[x_{1}, \ldots, x_{#1}]]}

\newcommand*{\pp}{\mathfrak{p}}

\newcommand*{\restr}[2]{{#1}|_{#2}}
\newcommand*{\Rinf}{R_{\infty}}

\newcommand*{\Sym}[1]{\mathrm{Sym}(#1)}
\newcommand*{\sympolnZ}{\polZ{n}^{\Sym{n}}}
\newcommand*{\sympolrZ}{\polZ{r}^{\Sym{r}}}
\newcommand*{\sympownZ}{\powZ{n}^{\Sym{n}}}
\newcommand*{\powinfZ}{\Z[[x_{1}, \ldots, x_{n}, \ldots]]}
\newcommand*{\Z}{\mathbb{Z}}
\newcommand*{\Znonneg}{\Z_{\geq 0}}
\newcommand*{\Zpos}{\Z_{>0}}

\renewcommand{\phi}{\varphi}

\let\originalleft\left
\let\originalright\right
\renewcommand{\left}{\mathopen{}\mathclose\bgroup\originalleft}
\renewcommand{\right}{\aftergroup\egroup\originalright}

\declaretheorem[style=definition, name = Definition, numberwithin=section]{defin}

\declaretheorem[name = Theorem, sibling=defin]{theorem}
\declaretheorem[name = Lemma, sibling=defin]{lemma}
\declaretheorem[name = Proposition, sibling=defin]{prop}
\declaretheorem[name = Corollary, sibling=defin]{cor}
\declaretheorem[style=remark, name = Remark, numbered = no]{remark}
\declaretheorem[name = Conjecture, numbered = no]{conjecture}
\declaretheorem[name = Theorem, numbered = no]{theorem*}
\declaretheorem[name = Objective, sibling=defin]{objective}

\numberwithin{equation}{section}

\crefname{objective}{Objective}{Objectives}
\crefname{prop}{Proposition}{Propositions}
\crefname{cor}{Corollary}{Corollaries}

\newcommand{\drawyoungdiagram}[1]{
	\begin{tikzpicture}[scale=0.5]
		\foreach \i [count=\j] in {#1}{
			\draw (0, -\j) grid ($({\i}, {-\j + 1})$);
		}
	\end{tikzpicture}
}
\title{Reidemeister spectra of free nilpotent groups and plethysms of Schur functions}
\author{Pieter Senden\footnote{E-mail: \url{pieter.senden@telenet.be}, ORCID: 0000-0002-3107-6775}}
\date{}
\begin{document}
\maketitle
\begin{abstract}
	We establish a strong link between two open problems: determining the Reidemeister spectrum of free nilpotent groups and determining the coefficients in the Schur expansion of plethysms of Schur functions.
	Specifically, we show that the expressions occurring in the computations for the Reidemeister spectrum are sums of plethysms of the form \(s_{1^{i}}[g]\), where \(g\) is a Schur function or a Schur positive function.
\end{abstract}
\let\thefootnote\relax\footnote{2020 \emph{Mathematics Subjects Classification}. Primary: 20E36, 05E05}
\let\thefootnote\relax

\section{Introduction}

Determining the Reidemeister spectra of free nilpotent groups and determining the coefficients in the Schur expansion of the plethysms of Schur functions are both well known and well studied open problems in their respective fields.
For both problems, significant progress has been made, but complete solutions still remain unknown.
In this article, we establish a strong link between both problems, in the sense that a significant step towards fully solving the problem of the Reidemeister spectra can be formulated as a problem of plethysms.
To establish this link, we travel from group theory to combinatorics, starting from Reidemeister spectra and ending at plethysms, passing through the realms of Lie algebras and symmetric polynomials along the way.

As this journey passes through several quite different mathematical fields, we want everything to be comprehensible for both those with little knowledge of Reidemeister numbers and for those with little knowledge of plethysms of Schur functions.
This means that at several points, we will explain ourselves more thoroughly than perhaps needed in the eyes of an expert in the field.

\subsection{Starting point: Reidemeister spectra of free nilpotent groups}
Let \(G\) be a group and \(\phi \in \End(G)\).
We say that \(x, y \in G\) are \emph{\(\phi\)-conjugate} if \(x = zy\inv{\phi(z)}\) for some \(z \in G\).
This defines an equivalence relation on \(G\) and the number of equivalence classes is called the \emph{Reidemeister number} of \(\phi\), which is denoted by \(R(\phi)\).
When \(\phi\) is not specified, we also speak of \emph{twisted conjugacy}.
Finally, the \emph{Reidemeister spectrum} of \(G\) is the set \(\SpecR(G) := \{R(\psi) \mid \psi \in \Aut(G)\}\).
If \(\SpecR(G) = \{\infty\}\), we say that \(G\) has the \emph{\(\Rinf\)-property}; if \(\SpecR(G) = \N_{0} \cup \{\infty\}\), we say that \(G\) has \emph{full Reidemeister spectrum}.

The concept of twisted conjugacy arises from algebraic topology, more specifically from Nielsen fixed-point theory, where its topological analog serves as an estimate for the number of fixed points of a continuous self map \cite{Jiang83}.
However, twisted conjugacy also occurs in representation theory \cite{Springer06}, Galois cohomology \cite{Berhuy10}, and the Arthur-Selberg trace formula \cite{Shokranian92}.

The last two decades, there has been strong algebraic interest in twisted conjugacy, especially in the \(\Rinf\)-property and in Reidemeister spectra of nilpotent groups.
For instance, nilpotent quotients of both surface groups \cite{DekimpeGoncalves16} and Baumslag-Solitar groups \cite{DekimpeGoncalves20}, and several \(2\)-step nilpotent quotients of right-angled Artin groups \cite{DekimpeLathouwers23} all have been studied.

There is particular interest in the free nilpotent groups.
Let \(r \geq 2\) and \(c \geq 1\) be integers.
Recall that the \emph{free \(c\)-step nilpotent group of rank \(r\)} is the group
\[
	N_{r, c} := \frac{F(r)}{\gamma_{c + 1}(F(r))},
\]
where \(F(r)\) is the free group of rank \(r\) and \(\gamma_{1}(G) := G, \gamma_{i + 1}(G) := [\gamma_{i}(G), G]\) for any group \(G\) and integer \(i \geq 1\).
For \(c = 1\), we obtain the free abelian groups and it is well known \cite[\S~3]{Romankov11} that
\[
	\SpecR(\Z^{r}) = \begin{cases}
 		\{2, \infty\}	&	\mbox{if } r = 1,	\\
 		\N_{0} \cup \{\infty\}	&	\mbox{otherwise.}
 \end{cases}
\]
The Reidemeister spectrum of (non-abelian) free nilpotent groups has been studied several times: in 2009, D.\;Gon\c{c}alves and P.\;Wong showed that \(N_{2, c}\) has the \(\Rinf\)-property for \(c \geq 9\) and that \(N_{r, 2}\) never has the \(\Rinf\)-property \cite{GoncalvesWong09}; in 2011, V.\;Roman'kov generalised the former by showing that \(N_{r, c}\) has the \(\Rinf\)-property if \(r \in \{2, 3\}\) and \(c \geq 4r\), and if \(r \geq 4\) and \(c \geq 2r\) \cite{Romankov11}.
Finally, in 2014, K.\;Dekimpe and D.\;Gon\c{c}alves classified all free nilpotent groups with the \(\Rinf\)-property:

\begin{theorem}[{\cite[Theorem~2.1]{DekimpeGoncalves14}}]	\label{theo:CharacterisationFreeNilpotentRinf}
	Let \(r \geq 2\) and \(c \geq 1\) be integers.
Then \(N_{r, c}\) has the \(\Rinf\)-property if and only if \(c \geq 2r\).
\end{theorem}

With the \(\Rinf\)-question answered, the focus has shifted to determining the Reidemeister spectrum in the remaining cases:
\begin{objective}
	Let \(r \geq 2\) and \(c \geq 1\) be integers.
Determine the Reidemeister spectrum of \(N_{r, c}\) if \(c < 2r\).
\end{objective}

Earlier, we showed that the case \(c = 1\) is fully done.
V.\;Roman'kov \cite{Romankov11} already determined that
\begin{align*}
	&\SpecR(N_{2, 2}) = 2\N_{0} \cup \{\infty\},	\\
	&\SpecR(N_{2, 3}) = \{2k^{2} \mid k \in \N_{0}\} \cup \{\infty\}	\quad \text{and} \\
	&\SpecR(N_{3, 2}) = \{2k - 1 \mid k \in \N_{0}\} \cup 4 \N_{0} \cup \{\infty\}.
\end{align*}
K.\;Dekimpe, S.\;Tertooy and A.\;Vargas completely solved the case \(c = 2\) and showed that \(N_{r, 2}\) has full Reidemeister spectrum if and only if \(r \geq 4\) \cite{DekimpeTertooyVargas20}.
They also show that \(\SpecR(N_{r, c})\) is never full if \(c \geq r\) and conjecture the following:
\begin{conjecture}[{\cite[Conjecture~5.2]{DekimpeTertooyVargas20}}]
	Let \(r \geq 2\) and \(c \geq 3\) be integers.
Then the Reidemeister spectrum of \(N_{r, c}\) is not full.
\end{conjecture}

Their proof for \(\SpecR(N_{r, 2})\) starts with a theoretical approach to the general problem of determining \(\SpecR(N_{r, c})\), after which they specify to the case \(c = 2\) and prove that \(\SpecR(N_{r, 2})\) is full for \(r \geq 4\).
They do so by giving for each positive integer \(n\) an automorphism with Reidemeister number equal to \(n\).
However, the computations are rather tedious and seem hard to generalise to other values of \(r\) and \(c\).

In this article, we continue along their theoretical approach.
They show that the Reidemeister number of an endomorphism \(\phi \in \End(N_{r, c})\) is fully determined by the eigenvalues of the induced map \(\phi_{1}\) on \(\frac{N_{r, c}}{\gamma_{2}(N_{r, c})}\) (since this quotient group is free abelian of finite rank, it makes sense to speak of eigenvalues).
While using this approach to determine \(\SpecR(N_{r, 2})\), they make use of the fact that \(R(\phi)\) is a symmetric polynomial in the eigenvalues of \(\phi_{1}\).
However, they never prove that this fact holds for all values of \(r\) and \(c\).
This result suggests a more combinatorial look at the problem of determining the Reidemeister spectrum of free nilpotent groups, which is the goal of this article.
Eventually, we relate this problem to a major open problem from combinatorics: the plethysm of Schur functions.

\subsection{Destination: plethysm of Schur functions}
The concepts of Schur functions and plethysms come from the field of symmetric functions.
Here, we give a rather informal introduction to these concepts; we work them out in more detail in \cref{sec:Combinatorics}.

A symmetric function can be thought of as an (infinite) power series in infinitely many variables \(x_{1}, x_{2}, \ldots\) such that permuting the variables does not change the power series.
For instance,
\[
	\sum_{i \geq 1} x_{i} \quad \text{ and } \quad \sum_{i < j} x_{i}x_{j}
\]
are both symmetric functions.
The Schur functions are a family of symmetric functions \(\{s_{\lambda}\}_{\lambda}\) that are indexed by the integer partitions, that is, the non-increasing sequences \(\lambda = (\lambda_{1}, \ldots, \lambda_{k}, \ldots)\) of non-negative integers with finite sum.
These Schur functions have nice combinatorial properties and each symmetric function is a (possibly infinite) linear combination of them.

Given two symmetric functions \(f\) and \(g\), the plethysm \(f[g]\) is essentially the symmetric function we obtain by substituting the monomials in \(g\) for the variables in \(f\), where a monomial is a product of one or more not necessarily distinct \(x_{i}\).
D.\;E.\;Littlewood introduced the concept of plethysm \cite{Littlewood50}.
One of the major open problems (\cref{obj:PlethysmSchurFunctions}) in the field of symmetric functions is to find a combinatorial interpretation of the coefficients \(a_{\mu, \lambda}^{\nu}\) in
\[
	s_{\lambda}[s_{\mu}] = \sum_{\nu} a_{\lambda, \mu}^{\nu} s_{\nu}.
\]
It is known that the coefficients \(a_{\lambda, \mu}^{\nu}\) are non-negative integers for all \(\lambda, \mu, \nu\), so it makes sense to look for a combinatorial interpretation for them.
No such interpretation is known in general, although it is known for certain special cases, such as \(s_{2}[s_{n}]\) \cite{Littlewood50} and \(s_{4}[s_{n}]\) \cite{Foulkes54,Howe87}.
Additionally, there are several algorithms known to compute these coefficients \cite{Agaoka95,Ibrahim56,Yang98}; see \cite{LoehrRemmel11} for an exposé on the topic and \cite{ColmenarejoOrellanaSaliolaSchillingZabrocki22} for a (non-exhaustive) list of known plethysms.

The coefficients \(a_{\lambda, \mu}^{\nu}\) also occur in representation theory.
Given a finite dimensional vector space \(V\), we look at the category of \(\GL(V)\)-modules.
On this category, we consider a family of functors \(\{S^{\lambda}\}_{\lambda}\), indexed by the partitions as well, called the Schur functors.
These have the property that for each partition \(\lambda\), the module \(S^{\lambda}V\) is irreducible; conversely, each irreducible \(\GL(V)\)-module is of the form \(S^{\lambda}V\) for some partition \(\lambda\).
Now, the \(\GL(V)\)-module \(S^{\lambda}(S^{\mu}V)\) decomposes as a direct sum of irreducible modules \(S^{\nu}V\), and the multiplicity of \(S^{\nu}V\) in this decomposition is precisely \(a_{\lambda, \mu}^{\nu}\).

This is one way of viewing the link between Schur functions and the representation theory of \(\GL(n, \C)\).
In \cref{subsec:SchurFunctionRepresentationTheory}, we present a different point of view.
We also scratch the surface of the link between Schur functions and the representation theory of the symmetric group \(\Sym{n}\).
We refer the reader to \cite{Weyman03} for more information about Schur functors, and to \cite{JamesKerber81,MacDonald95} for more information about representations of \(\Sym{n}\).
Finally, plethysms also occur in other areas, such as \(\tau\)-rings \cite{Hoffman79} and physics \cite{Wybourne70}.

We show that the problem of determining the Reidemeister spectrum of free nilpotent groups is related to plethysms of specific Schur functions, namely \(s_{1^{i}} [s_{k, 1}]\), and of other symmetric functions (see \cref{obj:CombinatorialObjectiveNrc,obj:CombinatorialObjectiveMrc}).

\section{From Reidemeister spectra to symmetric polynomials}	\label{sec:FromReidemeisterToPolynomials}

We start our journey with the main tools for computing Reidemeister numbers on abelian and nilpotent groups.
For ease of notation, we define the map
\[
	\inftynorm{\cdot}: \Z \to \N_{0} \cup \{\infty\}: n \mapsto \begin{cases}
		\infty	&	\mbox{if } n = 0,	\\
		|n|		&	\mbox{otherwise.}	
\end{cases}
\]
We refer to this maps as the \emph{\(\inftynorm{\cdot}\)-norm}.

\begin{lemma}[{\cite[\S~3]{Romankov11}}]	\label{lem:ReidemeisterNumberFreeAbelianGroups}
	Let \(A\) be a free abelian group of finite rank.
Let \({\phi \in \End(A)}\).
Then
	\[
		R(\phi) = \inftynorm{\det(\phi - \Id_{A})} = \inftynorm{p_{\phi}(1)},
	\]
	where \(p_{\phi}\) is the characteristic polynomial of \(\phi\).
\end{lemma}
The main tool for computing Reidemeister numbers on nilpotent groups is the so-called \emph{product formula}.
It was first discussed by V.\;Roman'kov \cite[Theorem~2.6, Lemma~2.7]{Romankov11}; see \cite[Proposition~5]{DekimpeGoncalvesOcampo21} for a statement and full proof of the result.
We say that a group \(G\) is \emph{nilpotent} if \(\gamma_{i}(G) = 1\) for some \(i \geq 1\), with \(\gamma_{i}(G)\) defined by \(\gamma_{1}(G) = G\) and \(\gamma_{i + 1}(G) := [\gamma_{i}(G), G]\) for \(i \geq 1\), as before.

\begin{theorem}[{\cite[Proposition~5]{DekimpeGoncalvesOcampo21}}]	\label{theo:productFormulaFGTFNilpotent}
	Let \(N\) be a \fgtf nilpotent group and let \(\phi \in \End(N)\).
Suppose that
	\[
		1 = N_{0} \lhd N_{1} \lhd \ldots \lhd N_{c} = N
	\]
	is a central series (\ie \([N_{i}, N] \leq N_{i - 1}\) for all \(i \in \{1, \ldots, c\}\)) such that
	\begin{itemize}
		\item for all \(i \in \{1, \ldots, c\}\), \(\phi(N_{i}) \leq N_{i}\), and
		\item for all \(i \in \{1, \ldots, c\}\), \(N_{i} / N_{i - 1}\) is torsion-free.
	\end{itemize}
	For \(i \in \{1, \ldots, c\}\), let \(\phi_{i}: N_{i} / N_{i - 1} \to N_{i} / N_{i - 1}\) denote the induced endomorphism.
Then
	\[
		R(\phi) = \prod_{i = 1}^{c} R(\phi_{i}).
	\]
\end{theorem}

Since the factors \(N_{i} / N_{i - 1}\) are finitely generated torsion-free abelian groups, the numbers \(R(\phi_{i})\) are given by \cref{lem:ReidemeisterNumberFreeAbelianGroups}, and are therefore completely determined by the eigenvalues of \(\phi_{i}\).

For free nilpotent groups, we can say even more.
Let \(r \geq 2\) and \(c \geq 2\) be integers.
For \(\phi \in \End(N_{r, c})\), we can apply \cref{theo:productFormulaFGTFNilpotent} by using the lower central series of \(N_{r, c}\), which is given by
\[
	1 = \gamma_{c + 1}(N_{r, c}) \lhd \gamma_{c}(N_{r, c}) \lhd \ldots \lhd \gamma_{1}(N_{r, c}) = N_{r, c}.
\]
As the name suggests, the lower central series is indeed a central series, and, moreover, it satisfies both conditions from \cref{theo:productFormulaFGTFNilpotent} (see \cite[Lemma~1.1.2]{Dekimpe96}).

Now, for any integer \(i \in \{2, \ldots, c\}\), the eigenvalues of \(\phi_{i}\) for \(i \geq 2\) are completely determined by those of \(\phi_{1}\).
To prove this, we largely follow the approach by K.\;Dekimpe, S.\;Tertooy and A.\;Vargas mentioned in the introduction.
However, our use of Lie algebras differs from theirs.

\begin{remark}
	Throughout this article, whenever we (implicitly or explicitly) apply \cref{theo:productFormulaFGTFNilpotent} to a free nilpotent group and an endomorphism \(\phi\), we use the lower central series.
\end{remark}

\begin{defin}
	Let \(G\) be a group.
The \emph{Lie ring associated to \(G\)} is the abelian group
	\[
		L(G) := \bigoplus_{i = 1}^{\infty} \frac{\gamma_{i}(G)}{\gamma_{i + 1}(G)}
	\]
	equipped with a Lie bracket defined as follows: for \(i, j \geq 1\), \(x \in \gamma_{i}(G)\) and \(y \in \gamma_{j}(G)\) arbitrary, we put
	\[
		[x\gamma_{i + 1}(G), y\gamma_{j + 1}(G)] := [x, y]\gamma_{i + j + 1}(G),
	\]
	and then extend it linearly to the whole of \(L(G)\).
Here, \([x, y]\) on the right-hand side denotes the commutator in \(G\), which is defined as \([x, y] := \inv{x}\inv{y} xy\).
The properties of the Lie bracket are:
	\begin{itemize}
		\item for all \(x \in L(G): [x, x] = 0\);
		\item for all \(x, y \in L(G): [x, y] = -[y, x]\);
		\item for all \(x, y, z \in L(G): [[x, y], z] + [[y, z], x] + [[z, x], y] = 0\) (Jacobi identity).
	\end{itemize}

\end{defin}

By tensoring with \(\C\), we can make \(L(G)\) into a (complex) Lie algebra \(L^{\C}(G) := L(G) \otimes_{\Z} \C\).
Every \(\phi \in \End(G)\) induces a Lie ring morphism on \(L(G)\) given by
\begin{equation}	\label{eq:definitionLieRingMorphism}
	\phi_{*}\left(\sum_{i \geq 1}x_{i}\gamma_{i + 1}(G)\right) := \sum_{i \geq 1} \phi(x_{i})\gamma_{i + 1}(G)
\end{equation}
where \(x_{i} \in \gamma_{i}(G)\) for each \(i \geq 1\) and only finitely many of them are non-trivial.
Due to the way the Lie bracket on \(L(G)\) is defined, \(\phi_{*}\) preserves the Lie bracket.
In addition, if, for every \(i \geq 1\), we write \(\phi_{i}\) for the induced endomorphism on \(\gamma_{i}(G) / \gamma_{i + 1}(G)\), we can rewrite \eqref{eq:definitionLieRingMorphism} as
\[
	\phi_{*}\left(\sum_{i \geq 1}x_{i}\gamma_{i + 1}(G)\right) = \sum_{i \geq 1} \phi_{i}(x_{i}\gamma_{i + 1}(G)).
\]

The Lie ring morphism \(\phi_{*}\) in turn induces a Lie algebra morphism on \(L^{\C}(G)\), which we also write as \(\phi_{*}\).

Let \(r \geq 2\) and \(c \geq 2\) be integers.
For \(G = N_{r, c}\), the associated Lie ring \(L(N_{r, c})\) is given by
\[
	L(N_{r, c}) = \bigoplus_{i = 1}^{c} \frac{\gamma_{i}(N_{r, c})}{\gamma_{i + 1}(N_{r, c})}
\]
as \(\gamma_{i}(N_{r, c}) = 1\) for \(i \geq c + 1\).
The Lie algebra \(L^{\C}(N_{r, c})\), in turn, is (isomorphic to) the \emph{free \(c\)-step nilpotent Lie algebra of rank \(r\)}, which is defined as the Lie algebra
	\[
		\nrc := \frac{\ff_{r}}{\gamma_{c + 1}(\ff_{r})}.
	\]	
	Here, \(\ff_{r}\) is the free (complex) Lie algebra of rank \(r\) and \(\gamma_{1}(\nn) = \nn\) and \(\gamma_{i + 1}(\nn) = [\gamma_{i}(\nn), \nn]\) for any Lie algebra \(\nn\) and integer \(i \geq 1\).
	
	One way to realise the isomorphism between \(L^{\C}(N_{r, c})\) and \(\nrc\) goes as follows: let \(X := \{X_{1}, \ldots, X_{r}\}\) be a set of \(r\) symbols.
Suppose \(F(r)\) is the free group on \(X\) and \(\ff_{r}\) is the free Lie algebra on \(X\).
Since \(X\) generates \(F(r)\), the set \(X_{r, c} := \{X_{i} \gamma_{c + 1}(F(r)) \mid i \in \{1, \ldots, r\}\}\) generates \(N_{r, c}\) as a group.
By \cite[17.2.1~Lemma]{KargapolovMerzljakov79}, each \(\gamma_{i}(N_{r, c})\) is generated by \(\gamma_{i + 1}(N_{r, c})\) and all \(i\)-fold commutator brackets of elements in \(X_{r, c}\).
The construction of the associated Lie ring then implies that \(L(N_{r, c})\) is generated, as a Lie ring, by the set \(\{X_{i} \gamma_{2}(N_{r, c}) \mid i \in \{1, \ldots, r\}\}\), which is the image of \(X_{r, c}\) under the natural projection to \(N_{r, c} / \gamma_{2}(N_{r, c})\).
Consequently, we can define a map \(\Psi: L^{\C}(N_{r, c}) \to \nn_{r, c}\) by putting \(\Psi(X_{i} \gamma_{2}(N_{r, c}) \otimes z) := zX_{i}\) for each \(i \in \{1, \ldots, r\}\) and \(z \in \C\), as their images completely determine \(\Psi\) on the whole of \(L^{\C}(N_{r, c})\).
This map \(\Psi\) is the desired isomorphism.
	
	This isomorphism \(\Psi\) enables us to, given \(\phi \in \End(N_{r, c})\), work on \(\nrc\) to compute the eigenvalues of \(\phi_{*}\).
Indeed, define \(\psi_{*} := \Psi \circ \phi_{*} \circ \inv{\Psi} \in \End(\nrc)\) and let \(\psi_{*, i}\) be the induced map on \(\gamma_{i}(\nrc) / \gamma_{i + 1}(\nrc)\).
We argue that \(\psi_{*, i}\) has the same eigenvalues as \(\phi_{i}\) for each \(i \in \{1, \ldots, c\}\).
	
	Fix \(i \in \{1, \ldots, c\}\).
The map \(\Psi\) induces the isomorphism
	\[
		\Psi_{i}: \frac{\gamma_{i}(L^{\C}(N_{r, c}))}{\gamma_{i + 1}(L^{\C}(N_{r, c}))} \to \frac{\gamma_{i}(\nrc)}{\gamma_{i + 1}(\nrc)}.
	\]
	By the definition of \(L(N_{r, c})\),
	\[
		\frac{\gamma_{i}(L(N_{r, c}))}{\gamma_{i + 1}(L(N_{r, c}))} = \frac{\gamma_{i}(N_{r, c})}{\gamma_{i + 1}(N_{r, c})},
	\]
	so by the properties of the tensor product,
	
	\[
		\frac{\gamma_{i}(L^{\C}(N_{r, c}))}{\gamma_{i + 1}(L^{\C}(N_{r, c}))} = \frac{\gamma_{i}(N_{r, c})}{\gamma_{i + 1}(N_{r, c})} \otimes \C.
	\]
	Hence, \(\Psi_{i}\) takes on the form
	\begin{equation}	\label{eq:isomorphismFactorsLCS}
		\Psi_{i}: \frac{\gamma_{i}(N_{r, c})}{\gamma_{i + 1}(N_{r, c})} \otimes \C \to \frac{\gamma_{i}(\nrc)}{\gamma_{i + 1}(\nrc)}: x \mapsto \Psi(x) \gamma_{i + 1}(\nrc).
	\end{equation}
	Now, by definition of \(\Psi\) and \(\psi_{*}\), the following diagram commutes:
\[
	\begin{tikzcd}
		&	\frac{\gamma_{i}(N_{r, c})}{\gamma_{i + 1}(N_{r, c})} \otimes \C	\ar[r, "\Psi_{i}", "\simeq"']	 \ar[d, "\phi_{i} \otimes \Id_{\C}"] & \frac{\gamma_{i}(\nrc)}{\gamma_{i + 1}(\nrc)} \ar[d, "\psi_{*, i}"]	\\
		&	\frac{\gamma_{i}(N_{r, c})}{\gamma_{i + 1}(N_{r, c})} \otimes \C	\ar[r, "\Psi_{i}", "\simeq"']	&	\frac{\gamma_{i}(\nrc)}{\gamma_{i + 1}(\nrc)}
	\end{tikzcd}
\]
		
Thus, the map \(\psi_{*, i}\) has the same eigenvalues as \(\phi_{i} \otimes \Id_{\C}\), which in turn has the same eigenvalues as \(\phi_{i}\).

To determine the eigenvalues of \(\phi_{i}\) and show that they are completely determined by those of \(\phi_{1}\), we need the concept of a Hall basis of a free Lie algebra.

\begin{defin}	\label{def:HallBasisFreeLieAlgebra}
	Let \(\ff_{r}\) be the free Lie algebra of rank \(r\) generated by the elements \(X_{1}, \ldots, X_{r}\).
A \emph{Hall basis} is a totally ordered vector space basis \(\HH\) of \(\ff_{r}\) which is constructed in an inductive way as \(\HH = \bigcup_{n = 1}^{\infty} \HH_{n}\):
	\begin{itemize}
		\item \(\HH_{1} := \{X_{1}, \ldots, X_{r}\}\) and \(X_{1} < \ldots < X_{r}\).
		\item For \(n \geq 2\), assume that \(\HH_{1}\) up to \(\HH_{n - 1}\) have been constructed and that \(\bigcup_{k = 1}^{n - 1} \HH_{k}\) has been given a total order.
Then \(\HH_{n}\) is defined as the set of all Lie brackets \([Y, Z]\) with \(Y \in \HH_{k}\) and \(Z \in \HH_{l}\) such that
		\begin{itemize}
			\item \(k + l = n\),
			\item \(Y < Z\),
			\item if \(Z = [Z_{1}, Z_{2}]\) for some \(Z_{i} \in \HH_{k_{i}}\), then \(Z_{1} \leq Y\).
		\end{itemize}
		\item Finally, we extend the order to \(\bigcup_{k = 1}^{n} \HH_{k}\) by picking any total order of \(\HH_{n}\) and imposing that \(X < Y\) for all \(X \in \HH_{k}\) and \(Y \in \HH_{n}\) with \(k < n\).
	\end{itemize}
\end{defin}

For instance, the first three parts of a Hall basis of \(\ff_{3}\) are given by
\[
	\HH_{1} = \{X_{1}, X_{2}, X_{3}\},
\]
\[
	\HH_{2} = \{[X_{1}, X_{2}], [X_{1}, X_{3}], [X_{2}, X_{3}]\}
\]
and
\begin{align*}
	\HH_{3}	&= \{[X_{1}, [X_{1}, X_{2}]], [X_{1}, [X_{1}, X_{3}]], [X_{2}, [X_{1}, X_{2}]], [X_{2}, [X_{1}, X_{3}]],	\\
			& \quad [X_{2}, [X_{2}, X_{3}]], [X_{3}, [X_{1}, X_{2}]], [X_{3}, [X_{1}, X_{3}]], [X_{3}, [X_{2}, X_{3}]]\}
\end{align*}
where the order on \(\HH_{2}\) and \(\HH_{3}\) is given by the order in which the elements are written down.
For more details and for a proof that \(\HH\) is indeed a basis, we refer the reader to \cite[Chapter~IV]{Serre65}.

The size of each \(\HH_{k}\) is determined by E.\;Witt:
\begin{theorem}[{\cite[Satz~3]{Witt37}}]	\label{theo:DimensionHallBasis}
	Let \(r \geq 2\) and \(k \geq 1\) be integers with \(k \leq r\).
Let \(\HH\) be a Hall basis of \(\ff_{r}\).
The size of \(\HH_{k}\) is given by
	\[
		N(r, k) := \frac{1}{k} \sum_{d \mid k} \mu(d) r^{k / d},
	\]
	where \(\mu: \Z_{>0} \to \{-1, 0, 1\}\) is the M\"{o}bius function, \ie
	\[
		\mu(d) = \begin{cases}
			1	&	\mbox{if \(d\) is the product of an even number of distinct prime numbers} \\
			-1	&	\mbox{if \(d\) is the product of an odd number of distinct prime numbers} \\
			0	&	\mbox{otherwise}.
		\end{cases}
	\]
\end{theorem}

Let \(\HH\) be a Hall basis of \(\ff_{r}\) and \(i \geq 1\).
By the properties of a free Lie algebra, the linear span of \(\HH_{i}\) fully lies in \(\{0\} \cup{(\gamma_{i}(\ff_{r}) \setminus \gamma_{i + 1}(\ff_{r}))}\) for every \(i \geq 1\).
Using this, it is readily verified that the projection of the subset \(\bigcup_{k = 1}^{c} \HH_{k}\) forms a (vector space) basis of \(\nrc\) for all \(c \geq 2\).
We call such a basis a \emph{Hall basis} of \(\nrc\) as well and we will also write \(\bigcup_{k = 1}^{c} \HH_{k}\) for this projection.

We formulate the next result as a separate lemma for reference purposes.
\begin{lemma}	\label{lem:basisFactorsLCSnrc}
	Let \(\HH = \bigcup_{k = 1}^{c} \HH_{k}\) be a Hall basis of \(\nrc\) and let \(k \geq 1\) be an integer.
Then the natural projections of the elements in \(\HH_{k}\) form a basis of \(\gamma_{k}(\nrc) / \gamma_{k + 1}(\nrc)\).
\end{lemma}
\begin{proof}
	Recall from \eqref{eq:isomorphismFactorsLCS} that
	\[
		\frac{\gamma_{k}(N_{r, c})}{\gamma_{k + 1}(N_{r, c})} \otimes \C \cong \frac{\gamma_{k}(\nrc)}{\gamma_{k + 1}(\nrc)}.
	\]
	Additionally, by the Third Isomorphism Theorem,
	\[
		\frac{\gamma_{k}(N_{r, c})}{\gamma_{k + 1}(N_{r, c})} \cong \frac{\gamma_{k}(F(r))}{\gamma_{k + 1}(F(r))}.
	\]
	Now, by \cite[Satz~4]{Witt37}, the latter group is free abelian of rank \(N(r, k)\).
	Therefore, the vector space
	\[
	 	\frac{\gamma_{k}(N_{r, c})}{\gamma_{k + 1}(N_{r, c})} \otimes \C	
	\]
	has dimension \(N(r, k)\).
	
	Since the natural projection of \(\HH_{k}\) generates \(\gamma_{k}(\nrc) / \gamma_{k + 1}(\nrc)\) and \(\HH_{k}\) contains \(N(r, k)\) elements, it forms a basis of \(\gamma_{k}(\nrc) / \gamma_{k + 1}(\nrc)\).
	
\end{proof}

Given \(\phi \in \End(N_{r, c})\), we now express, using a Hall basis, all eigenvalues of the induced maps \(\phi_{i}\) in terms of those of \(\phi_{1}\).
Let \(x_{1}, \ldots, x_{r}\) be \(r\) indeterminates.
Given a Hall basis \(\HH\) of \(\nrc\), we define a map \(\eta: \HH \to \Z[x_{1}, \ldots, x_{r}]\) as follows: for each \(i \in \{1, \ldots, r\}\), we put \(\eta(X_{i}) := x_{i}\).
This defines \(\eta\) on \(\HH_{1}\).
Now, suppose that \(\eta\) is defined on \(\bigcup_{i = 1}^{n - 1} \HH_{i}\) for some \(n \geq 2\).
Let \(X \in \HH_{n}\).
Then \(X = [Y, Z]\) for some \(Y \in \HH_{k}, Z \in \HH_{l}\) where \(k + l = n\).
We then define
\[
	\eta(X) := \eta(Y)\eta(Z).
\]

Thus, \(\eta([X_{1}, X_{2}]) = \eta(X_{1})\eta(X_{2}) = x_{1}x_{2}\), for example.

Lastly, for an \(r\)-tuple \(\lambda = (\lambda_{1}, \ldots, \lambda_{r}) \in \C^{r}\), we define two maps.
We define \(\ev_{\lambda}\) to be the evaluation map associated to \(\lambda\), \ie
\[
	\ev_{\lambda}: \Z[x_{1}, \ldots, x_{r}] \to \C: p(x_{1}, \ldots, x_{r}) \mapsto p(\lambda_{1}, \ldots, \lambda_{r}).
\]
We then define, given a Hall basis \(\HH\) of \(\nrc\), the map \(\eta_{\lambda}: \HH \to \C\) to be the composition \({\ev_{\lambda}} \circ \eta\).

\begin{lemma}[{\cite[Lemma~3.5]{DekimpeTertooyVargas20}}]	\label{lem:eigenvaluesphii}
	Let \(\phi \in \End(N_{r, c})\) and let, for \(i \in \{1, \ldots, c\}\), \(\phi_{i}\) denote the induced endomorphism on \(\gamma_{i}(N_{r, c}) / \gamma_{i + 1}(N_{r, c})\).
Let \(\lambda = (\lambda_{1}, \ldots, \lambda_{r})\) be the eigenvalues  of \(\phi_{1}\), where each eigenvalue is listed as many times as its multiplicity.
Let \(\HH\) be a Hall basis of \(\nrc\).
Then, for each \(k \in \{1, \ldots, c\}\), the eigenvalues of \(\phi_{k}\) are given by
	\begin{equation}	\label{eq:eigenvaluesphik}
		\{\eta_{\lambda}(X) \mid X \in \HH_{k}\}.
	\end{equation}
	Each eigenvalue is also listed as many times as its multiplicity.
	
\end{lemma}
\begin{proof}
	Let \(\phi_{*}\) denote the induced Lie algebra morphism on \(\nrc\).
Since the semisimple part of \(\phi_{*}\) is also a Lie algebra morphism on \(\nrc\) (see \cite[Corollary~2, p.~135]{Segal83}), we may assume that \(\phi_{*}\) itself is semisimple.
As discussed earlier, the eigenvalues of the induced map on \(\gamma_{k}(\nrc) / \gamma_{k + 1}(\nrc)\) are the same as those of \(\phi_{k}\).
	
	Since \(\phi_{*}\) is semisimple, we can find a basis of \(\nrc\) consisting of eigenvectors of \(\phi_{*}\).
In particular, we can find \(r\) eigenvectors \(E_{1}, \ldots, E_{r}\), with respective eigenvalues \(\lambda_{1}, \ldots, \lambda_{r}\) such that their images under the natural projection to \(\nrc / \gamma_{2}(\nrc)\) form a basis of this quotient space.
This implies that \(E_{1}, \ldots, E_{r}\) freely generate \(\nrc\).
Consequently, we can construct a Hall basis \(\HH'\) of \(\nrc\) with \(\HH'_{1} = \{E_{1}, \ldots, E_{r}\}\) (in that order).
	
	Next, define \(\eta': \HH' \to \Z[x_{1}, \ldots, x_{r}]\) analogously to \(\eta\) and put \(\eta'_{\lambda} := {\ev_{\lambda}} \circ \eta'\).
We prove that every \(X\) in \(\HH'\) is an eigenvector of \(\phi_{*}\) with eigenvalue \(\eta'_{\lambda}(X)\).
Indeed, by construction, this is true for each \(X \in \HH'_{1}\).
Inductively, suppose it holds for all elements in \(\bigcup_{i = 1}^{n - 1} \HH'_{i}\) for some \(n \geq 2\).
Let \(X \in \HH'_{n}\) be arbitrary.
Then \(X = [Y, Z]\) for some \(Y, Z \in \bigcup_{i = 1}^{n - 1} \HH'_{i}\).
By the induction hypothesis, \(\phi_{*}(Y) = \eta'_{\lambda}(Y) Y\) and \(\phi_{*}(Z) = \eta'_{\lambda}(Z) Z\).
Then
	\begin{align*}
		\phi_{*}(X)	&=	\phi_{*}([Y, Z])	\\
					&=	[\phi_{*}(Y), \phi_{*}(Z)]	\\
					&=	[\eta'_{\lambda}(Y) Y, \eta'_{\lambda}(Z) Z]	\\
					&=	\eta'_{\lambda}(Y) \eta'_{\lambda}(Z) [Y, Z]	\\
					&=	\eta'_{\lambda}(X) X.
	\end{align*}
	This proves that the Hall basis \(\HH'\) consists entirely of eigenvectors of \(\phi_{*}\).

	Finally, from \cref{lem:basisFactorsLCSnrc}, it follows that the eigenvalues of the induced map on \(\gamma_{k}(\nrc) / \gamma_{k + 1}(\nrc)\), and thus also those of \(\phi_{k}\), are precisely given by \(\{\eta'_{\lambda}(X) \mid X \in \HH'_{k}\}\), listed with multiplicity.
	
	To completely finish the proof, we have to argue that \(\{\eta'_{\lambda}(X) \mid X \in \HH'_{k}\}\) equals \eqref{eq:eigenvaluesphik}, for \(\HH\) and \(\HH'\) are not (necessarily) equal.
We proceed in three steps:
	
	\begin{enumerate}[\bfseries Step 1:]
	\item The expression \(\{\eta'_{\lambda}(X) \mid X \in \HH'_{k}\}\) does not depend on the order we choose in \(\HH'_{k}\) for \(k \geq 2\) in the construction of \(\HH'\).
Indeed, to prove that \(\{\eta'_{\lambda}(X) \mid X \in \HH'_{k}\}\) consists precisely of all eigenvalues of \(\phi_{*, k}\), we never use this order.
	
	\item Let \(f: \nrc \to \nrc\) be the isomorphism determined by \(f(X_{i}) = E_{i}\) for all \(i \in \{1, \ldots, r\}\).
	If we order the set \(f(\HH)\) by putting \(f(X) < f(Y)\) if and only if \(X < Y\) for all \(X, Y \in \HH\), then \(f(\HH)\) is also a Hall basis of \(\nrc\).
	By the previous step, replacing \(\HH'\) with \(f(\HH)\) does not change the expression \(\{\eta'_{\lambda}(X) \mid X \in \HH'_{k}\}\).
	
	\item Finally, we prove that \(\eta(X) = (\eta' \circ f)(X)\) for all \(X \in \HH\).
	For \(X \in \HH_{1}\), this is true by construction of all three maps.
	So, suppose it holds for all \(X \in \HH_{k}\) with \(k < n\).
	Let \(X = [Y, Z] \in \HH_{n}\) for some \(Y \in \HH_{k}\) and \(Z \in \HH_{l}\) with \(k, l < n\).
	Then
	\begin{align*}
		\eta(X) = \eta([Y, Z]) &= \eta(Y) \eta(Z)	\\
						&= \eta'(f(Y))\eta'(f(Z))	\\
						&= \eta'([f(Y), f(Z)])	\\
						&= \eta'(f([Y, Z]))	\\
						&= \eta'(f(X)).
	\end{align*}
	
	Therefore,
	\[
		\{\eta'_{\lambda}(X) \mid X \in \HH'_{k}\} = \{\eta'_{\lambda}(f(X)) \mid X \in \HH_{k}\} = \{\eta_{\lambda}(X) \mid X \in \HH_{k}\}.
	\]
\end{enumerate}

This finishes the proof.
\end{proof}

Combining this with \cref{theo:productFormulaFGTFNilpotent}, we find that, for integers \(r \geq 2\) and \(c \geq 2\) and for \(\phi \in \End(N_{r, c})\),
\[
	R(\phi) = \inftynormb{\prod_{k = 1}^{c} \prod_{X \in \HH_{k}}(1 - \eta_{\lambda}(X))}.
\]
For instance, on \(N_{3, 2}\), we find that
\[
	R(\phi) = \inftynorm{(1 - \lambda_{1})(1 - \lambda_{2})(1 - \lambda_{3})(1 - \lambda_{1}
	\lambda_{2})(1 - \lambda_{1}\lambda_{3})(1 - \lambda_{2}\lambda_{3})}.
\]

In other words, \(R(\phi)\) is (the \(\inftynorm{\cdot}\)-norm of) a polynomial in terms of the eigenvalues of \(\phi_{1}\), \ie there are polynomials \(f_{r, k} \in \Z[x_{1}, \ldots, x_{r}]\) for \(k \in \{1, \ldots, c\}\) such that
\[
	R(\phi) = \inftynormb{\prod_{k = 1}^{c} f_{r, k}(\lambda_{1}, \ldots, \lambda_{r})},
\]
namely
\[
	f_{r, k}(x_{1}, \ldots, x_{r}) = \prod_{X \in \HH_{k}}(1 - \eta(X)).
\]

A crucial observation is that these polynomials \(f_{r, k}\) do not depend on \(\phi\): the map \(\eta: \HH \to \Z[x_{1}, \ldots, x_{r}]\) completely determines which products of the \(x_{i}\) occur and it is independent of \(\lambda\).
We do, however, have to argue that the image \(\eta(\HH)\) does not depend on \(\HH\).

To that end, note that the proof of \cref{lem:eigenvaluesphii} from the second paragraph onwards holds for \emph{any} endomorphism \(\psi\) on \(\nrc\):
if \(\lambda_{1}, \ldots, \lambda_{r}\) are the eigenvalues of the induced endomorphism \(\psi_{1}: \nrc / \gamma_{2}(\nrc) \to \nrc / \gamma_{2}(\nrc)\), then, for each \(k \in \{1, \ldots, c\}\), those of \(\psi_{k} : \gamma_{k}(\nrc) / \gamma_{k + 1}(\nrc) \to \gamma_{k}(\nrc) / \gamma_{k + 1}(\nrc)\) are given by \eqref{eq:eigenvaluesphik}.
This implies that, for every \(k \geq 1\), the set
\[
	\{\eta_{\lambda}(X) \mid X \in \HH_{k}\} = \ev_{\lambda}(\{\eta(X) \mid X \in \HH_{k}\})
\]
is independent of \(\HH\) for all choices of \(\lambda \in \C^{r}\).
Indeed, every \(\lambda \in \C^{r}\) induces an endomorphism \(\psi\) of \(\nrc\) by putting \(\psi(Y_{i}) := \lambda_{i}Y_{i}\) for \(i \in \{1, \ldots, r\}\) for some freely generating set \(\{Y_{1}, \ldots, Y_{r}\}\) of \(\ff_{r}\).
As \(\C\) is an infinite field, this implies that \(\{\eta(X) \mid X \in \HH_{k}\}\) itself must be independent of \(\HH\), for all \(k \geq 1\).

Thus, if we know what the polynomials \(f_{r, k}\) are, we can compute \(\SpecR(N_{r, c})\) by running over all automorphisms \(\phi\) of \(N_{r, c}\), computing the eigenvalues of \(\phi_{1}\), and plugging those into the \(f_{r, k}\).

However, computing these eigenvalues is quite hard, since most of the time, there are no closed expressions for them.
To simplify the process of computing Reidemeister numbers, we argue that the \(f_{r, k}\) are symmetric polynomials.
A polynomial \(f \in \Z[x_{1}, \ldots, x_{n}]\) is called \emph{symmetric} if
	\[
		f(x_{1}, \ldots, x_{n}) = f(x_{\inv{\sigma}(1)}, \ldots, x_{\inv{\sigma}(n)})
	\]
	for each \(\sigma \in \Sym{n}\).
The ring of all symmetric polynomials in \(n\) variables with integer coefficients is denoted by \(\sympolnZ\).

To prove that the \(f_{r, k}\) are symmetric, let \(\psi \in \End(\nrc)\) and let \(\lambda = (\lambda_{1}, \ldots, \lambda_{r})\) be the eigenvalues of \(\psi_{1}\).
As argued before, the eigenvalues of \(\psi_{k}\) are given by
\[
	\{\eta_{\lambda}(X) \mid X \in \HH_{k}\} = \ev_{\lambda}(\{\eta(X) \mid X \in \HH_{k}\})
\]
Let \(\sigma \in \Sym{r}\) be a permutation.
Then \(\lambda_{\sigma} := (\lambda_{\sigma(1)}, \ldots, \lambda_{\sigma(r)})\) are still the eigenvalues of \(\psi_{1}\), only in a different order.
Consequently, the eigenvalues of \(\psi_{k}\) are also given by
\[
	\{\eta_{\lambda_{\sigma}}(X) \mid X \in \HH_{k}\} = \ev_{\lambda_{\sigma}}(\{\eta(X) \mid X \in \HH_{k}\}).
\]
Therefore,
\[
	\ev_{\lambda}(\{\eta(X) \mid X \in \HH_{k}\}) = \ev_{\lambda_{\sigma}}(\{\eta(X) \mid X \in \HH_{k}\}).
\]
This equality holds for all \(\sigma \in \Sym{r}\).
Moreover, it holds for all \(\lambda \in \C^{r}\) as well, since each \(\lambda \in \C^{r}\) induces an endomorphism of \(\nrc\) with eigenvalues given by \(\lambda\) (as argued earlier).

We can rewrite the previous equality as
\[
	\ev_{\lambda}(\{\eta(X) \mid X \in \HH_{k}\}) = \ev_{\lambda}(\inv{\sigma} \cdot \{\eta(X) \mid X \in \HH_{k}\}),
\]
where \(\tau \in \Sym{r}\) acts on \(\polZ{r}\) as follows:
\[
	\tau \cdot f(x_{1}, \ldots, x_{r}) := f(x_{\inv{\tau}(1)}, \ldots, x_{\inv{\tau}(r)}).
\]

Since, for all \(\sigma \in \Sym{r}\), this equality holds for all \(\lambda \in \C^{r}\) and since \(\C\) is an infinite field, we must have the equality
\[
	\{\eta(X) \mid X \in \HH_{k}\} = \inv{\sigma} \cdot \{\eta(X) \mid X \in \HH_{k}\},
\]
for all \(\sigma \in \Sym{r}\).
Equivalently, 
\[
	\{\eta(X) \mid X \in \HH_{k}\} = \sigma \cdot \{\eta(X) \mid X \in \HH_{k}\},
\]

This, finally, implies that
\[
	f_{r, k} = \prod_{X \in \HH_{k}}(1 - \eta(X))
\]
is a symmetric polynomial for all \(k \in \{1, \ldots, c\}\).
Thus, we have proven the following:

\begin{theorem}	\label{theo:ReidemeisterNumberSymmetricPolynomialEigenvalues}
	Let \(r \geq 2\) and \(c \geq 2\) be integers.
There exist symmetric polynomials \(f_{r, k} \in \sympolrZ\) (with \(1 \leq k \leq c\)) such that, for any \(\phi \in \End(N_{r, c})\),
	\[
		R(\phi) = \inftynormb{\prod_{k = 1}^{c} f_{r, k}(\lambda_{1}, \ldots, \lambda_{r})},
	\]
	where \(\lambda_{1}, \ldots, \lambda_{r}\) are the eigenvalues of \(\phi_{1}\).
\end{theorem}

For instance, for \(r = 3\) and \(k \in \{1, 2\}\), the polynomials are
\[
	f_{3, 1}(x_{1}, x_{2}, x_{3}) = \prod_{i = 1}^{3} (1 - x_{i})
\]
and
\[
	f_{3, 2}(x_{1}, x_{2}, x_{3})	= \prod_{1 \leq i < j \leq 3} (1 - x_{i}x_{j}).
\]

The next step is to use this observation to express the Reidemeister number of \(\phi\) purely in terms of the coefficients of the characteristic polynomial of \(\phi_{1}\).

\begin{defin}
	Let \(n, k \geq 1\) be integers with \(k \leq n\).
The \emph{elementary symmetric polynomial of degree \(k\) in \(n\) variables} is the symmetric polynomial
	\[
		e_{k}(x_{1}, \ldots, x_{n}) := \sum_{1 \leq i_{1} < \ldots < i_{k} \leq n} x_{i_{1}} \cdots x_{i_{k}}.
	\]
\end{defin}
Let \(p_{\phi} = \sum_{i = 0}^{r} a_{i}x^{i}\) be the characteristic polynomial of \(\phi_{1}\) with \(a_{r} = 1\) and let \(\lambda_{1}, \ldots, \lambda_{r}\) be the eigenvalues of \(\phi_{1}\), \ie the zeroes of \(p_{\phi}\).
Then
\[
	\sum_{i = 0}^{r} a_{i}x^{i} = \prod_{i = 1}^{r} (x - \lambda_{i}).
\]
Therefore, Vieta's relations yield
	\[
		a_{r - i} = (-1)^{i} e_{i}(\lambda_{1}, \ldots, \lambda_{r}).
	\]
for all \(i \in \{0, \ldots, r - 1\}\).
So, to express \(R(\phi)\) in terms of the coefficients of \(p_{\phi}\), it is sufficient to show that the polynomials \(f_{r, k}\) can be expressed in terms of the elementary symmetric polynomials in \(r\) variables.
This is possible due to the Fundamental Theorem of Symmetric Polynomials:

\begin{theorem}[{Fundamental theorem of symmetric polynomials, \cite[Theorem~1.12]{StewartTall02}}]	\label{theo:FundamentalTheoremSymmetricPolynomials}
	Let \(n \geq 1\) be an integer and let \(f \in \sympolnZ\).
Then there is a unique polynomial \(g \in \Z[x_{1}, \ldots, x_{n}]\) such that
	\[
		f(x_{1}, \ldots, x_{n}) = g(e_{1}(x_{1}, \ldots, x_{n}), \ldots, e_{n}(x_{1}, \ldots, x_{n})).
	\]
\end{theorem}

All together, we obtain
\begin{theorem}	\label{theo:ReidemeisterNumbersInTermsOfCoefficientCharPol}
	There exist polynomials \(\tilde{g}_{r, k} \in \Z[x_{1}, \ldots, x_{r}]\) for \(1 \leq k \leq c\) such that, for any \(\phi \in \End(N_{r, c})\), if \(p_{\phi} = \sum\limits_{i = 0}^{r} a_{i}x^{i}\) is the characteristic polynomial of \(\phi_{1} \in \End(N_{r, c} / \gamma_{2}(N_{r, c}))\) with \(a_{r} = 1\), then
	\[
		R(\phi) = \inftynormb{\prod_{k = 1}^{c} \tilde{g}_{r, k}(a_{0}, \ldots, a_{r - 1})}.
	\]
\end{theorem}
\begin{proof}
	By \cref{theo:ReidemeisterNumberSymmetricPolynomialEigenvalues},
	\[
		R(\phi) = \inftynormb{\prod_{k = 1}^{c} f_{r, k}(\lambda_{1}, \ldots, \lambda_{r})}
	\]
	for some symmetric polynomials \(f_{r, k} \in \sympolrZ\).
By \cref{theo:FundamentalTheoremSymmetricPolynomials}, there exist polynomials \(g_{r, k} \in \polZ{r}\) for \(k \in \{1, \ldots, c\}\) such that
	\[
		f_{r, k}(x_{1}, \ldots, x_{r}) = g_{r, k}(e_{1}(x_{1}, \ldots, x_{r}), \ldots, e_{r}(x_{1}, \ldots, x_{r})).
	\]
	Hence, by Vieta's relations mentioned earlier, 
	\begin{align*}
		f_{r, k}(\lambda_{1}, \ldots, \lambda_{r})	&= g_{r, k}((-1)^{1} a_{r - 1}, (-1)^{2} a_{r - 2}, \ldots, (-1)^{r} a_{0})	\\
										&= \tilde{g}_{r, k}(a_{0}, \ldots, a_{r - 1})
	\end{align*}
	where
	\[
		\tilde{g}_{r, k}(x_{1}, \ldots, x_{r}) := g_{r, k}((-1)^{1} x_{r}, (-1)^{2} x_{r - 1}, \ldots, (-1)^{r} x_{1}).
\qedhere
	\]
\end{proof}

To continue the example we gave earlier, we find that
\begin{align*}
	f_{3, 1}(x_{1}, x_{2}, x_{3})	&=	\prod_{i = 1}^{3} (1 - x_{i})	\\
							&=	1 - e_{1}(x_{1}, x_{2}, x_{3}) + e_{2}(x_{1}, x_{2}, x_{3}) - e_{3}(x_{1}, x_{2}, x_{3}), 
\end{align*}
so
\[
	g_{3, 1}(x_{1}, x_{2}, x_{3}) = 1 - x_{1} + x_{2} - x_{3}
\]
and
\[
	\tilde{g}_{3, 1}(x_{1}, x_{2}, x_{3}) = g_{3, 1}(-x_{3}, x_{2}, -x_{1}) = 1 + x_{1} + x_{2} + x_{3}.
\]
Similarly, 
\begin{align*}
	f_{3, 2}(x_{1}, x_{2}, x_{3})		&= \prod_{1 \leq i < j \leq 3} (1 - x_{i}x_{j})	\\
				&= 1 - e_{2}(x_{1}, x_{2}, x_{3}) + e_{1}(x_{1}, x_{2}, x_{3})e_{3}(x_{1}, x_{2}, x_{3}) - e_{3}(x_{1}, x_{2}, x_{3})^{2},
\end{align*}
so
\[
	g_{3, 2}(x_{1}, x_{2}, x_{3}) = 1 - x_{2} + x_{1}x_{3} - x_{3}^2
\]
and
\[
	\tilde{g}_{3, 2}(x_{1}, x_{2}, x_{3}) = g_{3, 2}(-x_{3}, x_{2}, -x_{1}) = 1 - x_{2} + x_{1}x_{3} - x_{1}^2.
\]
This theorem shows that there is actually no need to compute the eigenvalues of \(\phi_{*}\) in order to determine Reidemeister numbers on \(N_{r, c}\).
The Reidemeister number of \(\phi \in \End(N_{r, c})\) can be computed using only the coefficients of the characteristic polynomial \(p_{\phi}\) of \(\phi_{1}\).
K.\;Dekimpe, S.\;Tertooy and A.\;Vargas use this idea to determine the Reidemeister spectrum of \(N_{r, 2}\).

In addition, they elaborate on which characteristic polynomials can occur.
If \(\phi \in \Aut(N_{r, c})\), then all the induced maps \(\phi_{i}\) are  automorphisms as well.
Thus, we get a map
\[
	\omega: \Aut(N_{r, c}) \to \Aut(\Z^{r}): \phi \mapsto \phi_{1}.
\]

Since \(\phi_{1}\) is an automorphism and \(\Aut(\Z^{r}) \cong \GL(r, \Z)\), it follows that \(a_{0} = \det(\phi_{1}) = \pm 1\).
Thus, for every \(\phi \in \Aut(N_{r, c})\), the characteristic polynomial of \(\phi_{1} = \omega(\phi)\) is monic of degree \(r\) with integral coefficients and with constant coefficient equal to \(\pm 1\).

Conversely, they argue that \emph{every} monic polynomial of degree \(r\) with integral coefficients and with constant coefficient equal to \(\pm 1\) occurs as the characteristic polynomial of \(\omega(\phi)\) for some \(\phi \in N_{r, c}\).
To that end, we first note that the map \(\Aut(F_{r}) \to \Aut(\Z^{r})\), defined analogously to \(\omega\), is surjective (see \cite[Theorem~N4, Section~3.5]{KarrassMagnusSolitar76}).
Since every automorphism of \(F_{r}\) induces one on \(N_{r, c}\), the map \(\omega\) is surjective as well.

Secondly, given a polynomial
\[
	p(x) = \sum_{i = 0}^{r} a_{i} x^{i}
\]
with \(a_{i} \in \Z\) for all \(i \in \{1, \ldots, r\}\), \(a_{r} = 1\) and \(a_{0} = \pm 1\), its companion matrix
\[
	C(p) := 	\begin{pmatrix}
				0		&	0		&	\ldots	&	0		&	-a_{0}	\\
				1		&	0		&	\ldots	&	0		&	-a_{1}	\\
				0		&	1		&	\ldots	&	0		&	-a_{2}	\\
				\vdots	&	\vdots	&	\ddots	&	\vdots	&	\vdots	\\
				0		&	0		&	\ldots	&	1		&	-a_{r - 1}
			\end{pmatrix} \in \Z^{r \times r}
\]
has \(p(x)\) as characteristic polynomial.
Thus, given \(p(x)\), we obtain an element of \(\Aut(\Z^{r})\), which in turns yields an automorphism of \(N_{r, c}\) via \(\omega\).

Summarised, to compute the Reidemeister spectrum of \(N_{r, c}\), it is sufficient to determine the polynomials \(\tilde{g}_{r, k}\) from \cref{theo:ReidemeisterNumbersInTermsOfCoefficientCharPol} for \(k \in \{1, \ldots, c\}\) and then plug in all the possibilities of the coefficients of a monic polynomial of degree \(r\) with integer coefficients and with constant coefficient equal to \(\pm 1\).
Therefore, we shift our focus to determining these polynomials \(\tilde{g}_{r, k}\).

However, from the perspective of symmetric polynomials, it is more logical to find the polynomials \(g_{r, k}\) from the proof of \cref{theo:ReidemeisterNumbersInTermsOfCoefficientCharPol}, namely the polynomials expressing \(R(\phi)\) in terms of the elementary symmetric polynomials (evaluated in \(\lambda_{i}\)).
Thus, in the remainder of this article, we describe a method to determine the polynomials \(g_{r, k}\) and link it to the open problem of the plethysm of Schur functions, which we mentioned in the introduction.
We would like to mention that S.\;Tertooy has determined the polynomials \(g_{r, k}\) for several values of \(r\) and \(k\) in his PhD-thesis \cite[Chapter~5]{Tertooy19}.
\section{From symmetric polynomials to plethysms of Schur functions}	\label{sec:Combinatorics}

The example we gave throughout the previous section suggests that there is a difference between how \(f_{r, k}\) changes when \(r\) varies and when \(k\) varies, both how the expression for \(f_{r, k}\) itself and how the expansion in terms of elementary symmetric polynomials changes.
The table below shows the expressions for \(f_{r, k}\) and \(g_{r, k}\) for \(1 \leq r \leq 3\) and \(1 \leq k \leq 2\).
We use \(e_{i}\) for the variables of \(g_{r, k}\) to emphasise that \(g_{r, k}\) is the expansion of \(f_{r, k}\) in terms of elementary symmetric polynomials.

\begin{table}[h]
\centering
\begin{tabularx}{\textwidth}{>{\(}l<{\)} | >{\(}l<{\)} | >{\(}l<{\)} || >{\(}l<{\)} | >{\(}X<{\)}}
		r	&	f_{r, 1}					&	g_{r, 1}							&	f_{r, 2}											&	g_{r, 2}		\\\hline
		2	&	\displaystyle\prod_{i = 1}^{2}(1 - x_{i})&	1 - e_{1} + e_{2}					&	\displaystyle\prod_{1 \leq i < j \leq 2}	(1 - x_{i}x_{j})		&	1 - e_{2}	\\
		3	&	\displaystyle\prod_{i = 1}^{3}(1 - x_{i})&	1 - e_{1} + e_{2} - e_{3}			&	\displaystyle\prod_{1 \leq i < j \leq 3}	(1 - x_{i}x_{j})		&	1 - e_{2} + e_{1}e_{3} - e_{3}^{2}	\\
		4	&	\displaystyle\prod_{i = 1}^{4}(1 - x_{i})&	1 - e_{1} + e_{2} - e_{3}	+ e_{4}	&	\displaystyle\prod_{1 \leq i < j \leq 4}	(1 - x_{i}x_{j})		&	\begin{aligned}&1 - e_{2} + e_{1}e_{3} - e_{3}^{2} - e_{4} \\&- e_{1}^{2}e_{4} + 2e_{2}e_{4} + e_{1}e_{3}e_{4} \\&- e_{4}^{2} - e_{2}e_{4}^{2} + e_{4}^{3}\end{aligned}
\end{tabularx}
\caption{Expressions for \(f_{r, k}\) and \(g_{r, k}\) for small values of \(r\) and \(k\).}
\end{table}

We see in the table that an increment in \(r\) only yields extra terms in the expansion of \(g_{r, k}\).
More precisely, the table suggests that, given \(g_{r, k}\), we can obtain \(g_{r - 1, k}\) by substituting \(e_{r} = 0\).
However, for increasing \(k\), the expression for \(f_{r, k}\) becomes more involved, as
\[
	f_{r, 3} = \prod_{\substack{1 \leq i, j \leq r \\ i \ne j}} (1 - x_{i}^{2}x_{j}) \left(\prod_{1 \leq i < j < k \leq r} (1 - x_{i}x_{j}x_{k}) \right)^{2},
\]
for instance.
In turn, complexity of the expression for \(g_{r, k}\) also increases.

To grasp this behaviour all at once for all \(r\), we, informally speaking, take the limit \(r \to \infty\).
We make this more precise and rigorous by continuing our journey into the world of symmetric functions.

\subsection{Symmetric functions}
In this section, we introduce most of the necessary concepts concerning symmetric functions.
Our exposition is loosely based on the one by R.\;Stanley \cite[\S~7]{Stanley12a}, but differs in some places.

Let \(A\) be the set of all sequences of non-negative integers whose support is finite, \ie
\[
	A := \{\alpha = (\alpha_{i})_{i \geq 1} \mid \forall i \geq 1: \alpha_{i} \in \Znonneg, \text{ \(\supp(\alpha)\) finite}\}.
\]
The group \(\Sym{\Zpos}\) acts in a natural way on \(A\):
\[
	\Sym{\Zpos} \times A \to A: (\sigma, \alpha) \mapsto \sigma \cdot \alpha := (\alpha_{\inv{\sigma}(i)})_{i \geq 1}.
\]

Let \(x = (x_{1}, \ldots, x_{n}, \ldots)\) be a list of formal variables.
Given \(\alpha \in A\), we define
\[
	x^{\alpha} := \prod_{i \geq 1} x_{i}^{\alpha_{i}}.
\]
Since \(\alpha\) has finite support, this product is a well-defined monomial.
We define the \emph{ring of formal power series in a countable infinite number of variables} to be the set of all expressions of the form
	\begin{equation}	\label{eq:expressionFormalPowerSeries}
		f = \sum_{\alpha \in A} c_{\alpha} x^{\alpha},
	\end{equation}
where \(c_{\alpha} \in \Z\) for all \(\alpha \in A\).
To make this set into a ring, we equip it with pointwise addition and with the following multiplication:
	\[
		\sum_{\alpha \in A} c_{\alpha} x^{\alpha} \cdot \sum_{\alpha \in A} d_{\alpha} x^{\alpha} := \sum_{\alpha \in A} \sum_{\substack{\beta, \gamma \in A \\ \beta + \gamma = \alpha}} (c_{\beta} + d_{\gamma}) x^{\alpha}.
	\]
	
	Since for each \(\alpha \in A\) there are only finitely many couples \((\beta, \gamma) \in A \times A\) with \(\beta + \gamma = \alpha\), this product is well defined.
We let \(\powinfZ\) denote this ring.
For \(\alpha \in A\) and \(f \in \powinfZ\) as in \eqref{eq:expressionFormalPowerSeries}, we also write \([x^{\alpha}]f := c_{\alpha}\).

	A formal power series is called a \emph{symmetric function} if, in addition, \(c_{\alpha} = c_{\sigma \cdot \alpha}\) for all \(\alpha \in A\) and \(\sigma \in \Sym{\Zpos}\).
Since any two symmetric functions can be added and multiplied and the result is again a symmetric function, the symmetric functions form a subring of \(\powinfZ\), which is denoted by \(\Lambda\).

The \emph{degree} of a monomial \(x^{\alpha}\) is given by \(\deg(x^{\alpha}) := \sum_{i \geq 1} \alpha_{i}\) and the degree of a symmetric function \(f\) is
\[
	\deg(f) := \sup\{\deg(x^{\alpha}) \mid [x^{\alpha}]f \ne 0\}.
\]
Note that, with this definition, \(\deg(f) = \infty\) is possible, and this is where our definition differs from the one by R.\;Stanley; in his definition, symmetric functions always have a finite degree.
If \(n := \deg(f) = \deg(x^{\alpha})\) for all \(\alpha \in A\) with \([x^{\alpha}]f \ne 0\), then we call \(f\) \emph{homogeneous of degree \(n\)}.
Given a symmetric function \(f = \sum_{\alpha \in A} c_{\alpha} x^{\alpha}\), we can write it as the (infinite) sum of homogeneous functions \(f_{n}\) given by
\[
	f_{n} := \sum_{\substack{\alpha \in A \\ \deg(x^{\alpha}) = n}} c_{\alpha} x^{\alpha}.
\]

To illustrate these definitions, consider the power series
\[
	f = \sum_{i} x_{i}
\]
and
\[
	g = \sum_{i < j} x_{i}x_{j}^2.
\]
Then \(f\) is symmetric, whereas \(g\) is not, since \([x_{1}x_{2}^2]g = 1\), but \([x_{1}^{2}x_{2}]g = 0\).
Both \(f\) and \(g\) are homogeneous, with \(f\) having degree \(1\) and \(g\) degree \(3\).
Their product is given by
\[
	fg = \sum_{i < j} x_{i}^{2}x_{j}^{2} + \sum_{i < j} x_{i}x_{j}^{3} + \sum_{k \ne i < j \ne k} x_{i}x_{j}^{2}x_{k}.
\]

Since we want to be able to go back from symmetric functions to symmetric polynomials, we have to construct a way to do so.
For every integer \(n \geq 0\), we define
\[
	A_{n} := \{\alpha \in A \mid \supp(\alpha) \subseteq \{1, \ldots, n\}\}.
\]
Note that \(A_{0} = \{(0,\ldots, 0, \ldots)\}\) and that \(A_{n} \subseteq A_{n + 1}\) for every \(n \geq 0\).

\begin{defin}
	Let \(n \geq 1\) be an integer.
The \emph{ring of formal power series in \(n\) variables} is the set of all power series of the form
	\[
		f = \sum_{\alpha \in A_{n}} c_{\alpha} x^{\alpha}.
	\]
	where \(c_{\alpha} \in \Z\) for all \(\alpha \in A_{n}\), equipped with the same addition and multiplication as \(\powinfZ\).
We let \(\powZ{n}\) denote this ring.
\end{defin}

The group \(\Sym{n}\) acts in a similar way on \(A_{n}\) as \(\Sym{\Zpos}\) does on \(A\):
\[
	\Sym{n} \times A_{n} \to A_{n}: (\tau, \alpha) \mapsto \tau \cdot \alpha,
\]
where
\[
	(\tau \cdot \alpha)_{i} :=	\begin{cases}
									\alpha_{\inv{\tau}(i)}	&	\mbox{if } i \leq n	\\
									\alpha_{i}				&	\mbox{if } i > n.
								\end{cases}
\]
Hence, we can also define the \emph{ring of symmetric formal power series in \(n\) variables}, which is denoted by \(\sympownZ\).
Note that we can view \(\sympolnZ\), the ring of symmetric polynomials in \(n\) variables, as a subring of \(\sympownZ\).

\begin{defin}
	Let \(n \geq 1\) be an integer.
We define the map
	\[
		\Phi_{n}: \powinfZ \to \powZ{n}: f = \sum_{\alpha \in A} c_{\alpha} x^{\alpha} \mapsto \Phi_{n}(f) := \sum_{\alpha \in A_{n}} c_{\alpha} x^{\alpha}.
	\]
\end{defin}

In other words, \(\Phi_{n}\) `forgets' all monomials that contain a variable \(x_{i}\) with \(i > n\).
We will sometimes write \(f(x_{1}, \ldots, x_{n})\) to denote \(\Phi_{n}(f)\).
Clearly, \(\Phi_{n}\) maps \(\Lambda\) into \(\sympownZ\), but for certain \(f \in \Lambda\), \(\Phi_{n}(f)\) is even an element of \(\sympolnZ\), for all \(n \geq 1\).
For instance,
\[
	\Phi_{n}\left(\sum_{i} x_{i}\right) = \sum_{i = 1}^{n} x_{i} \in \sympolnZ
\]
for all \(n \geq 1\).
However,
\[
	\Phi_{n}\left(\sum_{i, j} x_{i}^{j} \right) = \sum_{i = 1}^{n} \sum_{j} x_{i}^{j} \notin \sympolnZ.
\]

\begin{lemma}	\label{lem:CharacterisationSymmetricFunctions}
	Let \(f \in \powinfZ\).
Then \(f \in \Lambda\) if and only if there exists an \(M \geq 1\) such that \(\Phi_{n}(f) \in \sympownZ\) for all \(n \geq M\).
\end{lemma}
\begin{proof}
	The `only if' is straightforward by taking \(M = 1\).
So, suppose that \(f \in \powinfZ\) is such that \(\Phi_{n}(f) \in \sympownZ\) for all \(n \geq M\) for some \(M \geq 1\).
Write \(f = \sum_{\alpha \in A} c_{\alpha} x^{\alpha}\).
For every \(n \geq M\), \(\Phi_{n}(f)\) is symmetric; since
	\[
		\Phi_{n}(f) = \sum_{\beta \in A_{n}} c_{\beta} x^{\beta},
	\]
	this means that \(c_{\beta} = c_{\tau \cdot \beta}\) for all \(\beta \in A_{n}\) and \(\tau \in \Sym{n}\).
Now, let \(\alpha \in A\) and \(\sigma \in \Sym{\Zpos}\) be arbitrary.
Since \(\alpha\) has finite support, so does \(\sigma \cdot \alpha\).
Hence, there exists an \(n \geq M\) such that \(\supp(\alpha) \cup \supp(\sigma \cdot \alpha) \subseteq \{1, \ldots, n\}\).
In particular, \(\alpha \in A_{n}\).
We now construct a permutation \(\tau\) of \(\{1, \ldots, n\}\) such that \(\tau \cdot \alpha = \sigma \cdot \alpha\).
Since
	\begin{align*}
		\supp(\sigma \cdot \alpha)	&=	\{i \geq 1 \mid (\sigma \cdot \alpha)_{i} \ne 0\}	\\
									&=	\{i \geq 1 \mid \alpha_{\inv{\sigma}(i)} \ne 0\}		\\
									&=	\sigma(\supp(\alpha)),
	\end{align*}
	and thus \(\sigma(\supp(\alpha)) \subseteq \{1, \ldots, n\}\), we can define \(\tau(i) := \sigma(i)\) for every \(i \in \supp(\alpha)\) to get a bijective map \(\tau: \supp(\alpha) \to \sigma(\supp(\alpha))\).
Since both the domain and the codomain are subsets of \(\{1, \ldots, n\}\) of the same size, we can extend this map to a bijection on \(\{1, \ldots, n\}\).
This can be done in various ways, but we just fix one, which we also write as \(\tau\).
We claim that \(\sigma \cdot \alpha = \tau \cdot \alpha\).
Let \(i \geq 1\) be arbitrary.
First, suppose that \((\sigma \cdot \alpha)_{i} \ne 0\).
Then \(\inv{\sigma}(i) \in \supp(\alpha)\).
Hence, \(\tau(\inv{\sigma}(i)) = \sigma(\inv{\sigma}(i)) = i\).
Consequently, \(\inv{\tau}(i) = \inv{\sigma}(i)\) and
	\[
		(\tau \cdot \alpha)_{i} = \alpha_{\inv{\tau}(i)} = \alpha_{\inv{\sigma}(i)} = (\sigma \cdot \alpha)_{i}.
	\]
	Conversely, suppose that \((\tau \cdot \alpha)_{i} \ne 0\).
Then \(\inv{\tau}(i) \in \supp(\alpha)\).
So, \(i = \tau(\inv{\tau}(i)) = \sigma(\inv{\tau}(i))\), which again implies that \(\inv{\tau}(i) = \inv{\sigma}(i)\).
Consequently,
	\[
		(\sigma \cdot \alpha)_{i} = \alpha_{\inv{\sigma}(i)} = \alpha_{\inv{\tau}(i)} = (\tau \cdot \alpha)_{i}.
	\]
	
	Summarised, we have proven that \(\sigma \cdot \alpha\) and \(\tau \cdot \alpha\) have the same support and moreover agree on these supports.
Hence, they are equal.
	
	Since \(\alpha \in A_{n}, \tau \in \Sym{n}\) and \(\sigma \cdot \alpha = \tau \cdot \alpha\), the equalities
	\[
		c_{\alpha} = c_{\tau \cdot \alpha} = c_{\sigma \cdot \alpha}
	\]
	hold.
As \(\alpha \in A\) and \(\sigma \in \Sym{\Zpos}\) are arbitrary, we conclude that \(f\) is symmetric.
\end{proof}
\begin{lemma}	\label{lem:equalityConditionPowerSeries}
	Let \(f, g \in \powinfZ\).
Then \(f = g\) if and only if there exists an \(M \geq 1\) such that \(\Phi_{n}(f) = \Phi_{n}(g)\) for all \(n \geq M\).
\end{lemma}
\begin{proof}
	The `only if' is trivial.
	So, suppose that \(f, g \in \powinfZ\) and \(M \geq 1\) are such that \(\Phi_{n}(f) = \Phi_{n}(g)\) for all \(n \geq M\).
	Write \(f = \sum_{\alpha \in A} c_{\alpha} x^{\alpha}\) and \(g = \sum_{\alpha \in A} d_{\alpha} x^{\alpha}\).
	Let \(\alpha \in A\).
	Since \(\alpha\) has finite support, there is an \(n \geq M\) such that \(\alpha \in A_{n}\).
	The equality \(\Phi_{n}(f) = \Phi_{n}(g)\) then implies that \(c_{\alpha} = d_{\alpha}\).
	As \(\alpha\) is arbitrary, the result follows.
\end{proof}

The elementary symmetric polynomials have their counterpart in the realm of symmetric functions.
To introduce the latter, we need the concept of a partition.

\begin{defin}
	Let \(n \geq 1\) be an integer.
A \emph{partition} of \(n\) is a non-increasing sequence \(\lambda = (\lambda_{1}, \ldots, \lambda_{k}, \ldots)\) of non-negative integers whose sum equals \(n\).
We write \(\lambda \vdash n\).
	
	The set of all partitions of \(n\) is denoted by \(\Par(n)\), the set of all partitions by \(\Par\).
\end{defin}

\begin{defin}
	Let \(n \geq 1\) and let \(\lambda \in \Par\) be a partition.
We define the \emph{elementary symmetric functions} as
	\[
		e_{n} := \sum_{i_{1} < \ldots < i_{n}} x_{i_{1}} \cdots x_{i_{n}}
	\]
	and
	\[
		e_{\lambda} := \prod_{i \geq 1} e_{\lambda_{i}}.
	\]
\end{defin}
Since \(\lambda\) has finite support, this product is well defined.
Also, the elementary symmetric functions are mapped to the elementary symmetric polynomials under \(\Phi_{n}\), \ie
\[
	\Phi_{n}(e_{m}) = e_{m}(x_{1}, \ldots, x_{n})
\]
for all \(1 \leq m \leq n\), and \(\Phi_{n}(e_{m}) = 0\) for \(m > n\).
Here, we see that in fact \(\Phi_{n}(e_{m}) \in \sympolnZ\).

\begin{theorem}	\label{theo:SymmetricFunctionsInTermsOfElementarySymmetricFunctions}
	Each symmetric function \(f \in \Lambda\) can be written as a series of elementary symmetric functions, \ie there exist \(c_{\lambda} \in \Z\) with \(\lambda \in \Par\) such that
	\[
		f = \sum_{\lambda \in \Par} c_{\lambda} e_{\lambda}.
	\]
\end{theorem}
\begin{proof}

	Write \(f = \sum_{n \geq 1} f_{n}\) in its homogeneous parts.
In \cite[7.4.4.~Theorem]{Stanley12a}, it is proven that each \(f_{n}\) can be written as
	\[
		f_{n} = \sum_{\lambda \in \Par(n)} c_{\lambda} e_{\lambda}
	\]
	for some \(c_{\lambda} \in \Z\).
By combining all of these expressions, we obtain the series for \(f\).
\end{proof}

We next properly define the notion of plethysm.
We remark that there are other equivalent ways of defining the plethysm of two symmetric functions, see \eg \cite[\S~2]{LoehrRemmel11} and \cite[\nopp I.8]{MacDonald95}.
The definition below suits our purposes the best.

\begin{defin}
	Let \(f, g\) be two symmetric functions.
Suppose that \(g = \sum_{i \geq 1} x^{\alpha_{i}}\) where \(\alpha_{i} \in A\) for each \(i \geq 1\).
The \emph{plethysm \(f[g]\)} is then defined as
	\[
		f[g] := f(x^{\alpha_{1}}, x^{\alpha_{2}}, \ldots).
	\]
\end{defin}

It is important here to note that we write \(g\) as a sum of \emph{monomials}, so that one monomial can occur several times in this summation.
For instance, 
\begin{align*}
	e_{2}[2e_{1}]	&=	e_{2}\left[x_{1} + x_{1} + x_{2} + x_{2} + \ldots\right]	\\
					&=	e_{2}(x_{1}, x_{1}, x_{2}, x_{2}, \ldots)	\\
					&=	\sum_{i \geq 1} x_{i}^{2} + 4 \sum_{i < j} x_{i}x_{j},
\end{align*}
whereas \(e_{2}[e_{1}]\) is just \(e_{2}\).
Thus, \(e_{2}[2e_{1}] \ne 2e_{2}[e_{1}]\).

\subsection{From symmetric polynomials to symmetric functions}	\label{subsec:FromSymmetricPolynomialsToSymmetricFunctions}
The polynomials \(f_{r, k}\) from \cref{theo:ReidemeisterNumberSymmetricPolynomialEigenvalues} are often given as a product, \eg
\[
	f_{r, 2} = \prod_{1 \leq i < j \leq r} (1 - x_{i} x_{j}).
\]

To move from symmetric polynomials to symmetric functions, we want to make sense of expressions of the form
\[
	\prod_{i < j} (1 - x_{i} x_{j}).
\]
Intuitively, we do so by working out the product distributively, but choosing in all but finitely many factors the term \(1\).
More rigorously, we proceed as follows:

\begin{defin}
	Let \(\{x^{\alpha_{i}}\}_{i \geq 1}\) be a collection of monomials.
We define
	\[
		\prod_{i \geq 1} (1 - x^{\alpha_{i}}) := 1 + \sum_{k = 1}^{\infty} (-1)^{k} \sum_{i_{1} < \ldots < i_{k}} x^{\alpha_{i_{1}}} \ldots x^{\alpha_{i_{k}}},
	\]
	provided all the latter infinite sums are well defined.
If all the infinite sums are well defined, we say the product \emph{converges}.
\end{defin}

For example, the product
\[
	\prod_{i} (1 - x_{i})
\]
equals
\[
	1 + \sum_{k = 1}^{\infty} (-1)^{k} \sum_{i_{1} < \ldots < i_{k}} x_{i_{1}} \ldots x_{i_{k}} = 1 + \sum_{k = 1}^{\infty} (-1)^{k} e_{k}.
\]

On the other hand, the product
\[
	\prod_{i} (1 - x_{1})
\]
is not defined, for the sum for \(k = 1\) equals \(\sum_{i_{1}} x_{1}\).

\begin{prop}	\label{prop:Product1-MonomialsEqualsAlternatingPlethysmElementary}
	Let \(\{x^{\alpha_{i}}\}_{i \geq 1}\) be a collection of monomials such that
	\[
		\prod_{i \geq 1} (1 - x^{\alpha_{i}})
	\]
	converges and
	\[
		g := \sum_{i \geq 1} x^{\alpha_{i}}
	\]
	is well defined.
Suppose both are symmetric as well.
Then
	\[
		\prod_{i \geq 1} (1 - x^{\alpha_{i}}) = \sum_{n = 0}^{\infty} (-1)^{n} e_{n}[g].
	\]
\end{prop}
\begin{proof}
	On both sides, we determine which products of monomials occur.
Let \(k \geq 1\) and \(1 \leq i_{1} <  \ldots < i_{k}\) be arbitrary.
On the left-hand side, we get the term \((-1)^{k} \prod_{j = 1}^{k} x^{\alpha_{i_{j}}}\).
This same term occurs on the right-hand side in the expansion of \((-1)^{k}e_{k}[g]\), since \(x_{i_{1}} \cdots x_{i_{k}}\) is a monomial occurring in \(e_{k}\) and the plethysm \(e_{k}[g]\) substitutes \(x_{i_{j}} = x^{\alpha_{i_{j}}}\) for \(j \in \{1, \ldots, k\}\).
	
	Conversely, let \(n \geq 1\) be arbitrary and consider \((-1)^{n} e_{n}[g]\).
Every term in this plethysm is of the form \((-1)^{n}x^{\alpha_{i_{1}}} \cdots x^{\alpha_{i_{n}}}\) for some \(1 \leq i_{1} < \ldots < i_{n}\), which clearly occurs on the left-hand side as well.
\end{proof}

We construct an infinite analogue of the Hall basis.
Let \(\ff_{\infty}\) be the free Lie algebra of countable infinite rank and suppose it is freely generated by the elements \(X_{1}, \ldots, X_{n}, \ldots\).
We construct a totally ordered set, which we will call a \emph{Hall set} \(\HHinf{}\), in a similar way as we constructed the Hall basis of \(\ff_{r}\), namely in an inductive way as \(\HHinf{} = \bigcup_{n = 1}^{\infty} \HHinf{n}\):
	\begin{itemize}
		\item \(\HHinf{1} := \{X_{1}, \ldots, X_{n}, \ldots\}\) and \(X_{1} < \ldots < X_{n} < \ldots\)
		\item For \(n \geq 2\), assume that \(\HHinf{1}\) up to \(\HHinf{n - 1}\) have been constructed and that \(\bigcup_{k = 1}^{n - 1} \HHinf{k}\) has been given a total order.
Then \(\HHinf{n}\) is defined as the set of all Lie brackets \([Y, Z]\) with \(Y \in \HHinf{k}\) and \(Z \in \HHinf{l}\) such that
		\begin{itemize}
			\item \(k + l = n\),
			\item \(Y < Z\),
			\item if \(Z = [Z_{1}, Z_{2}]\) for some \(Z_{i} \in \HHinf{k_{i}}\), then \(Z_{1} \leq Y\).
		\end{itemize}
		\item Finally, we extend the order to \(\bigcup_{k = 1}^{n} \HHinf{k}\) by picking any total order of \(\HHinf{n}\) and imposing that \(X < Y\) for all \(X \in \HHinf{k}\) and \(Y \in \HHinf{n}\) with \(k < n\).
	\end{itemize}

It is clear that, by construction, a Hall set on \(\ff_{\infty}\) projects down to a Hall basis of \(\ff_{r}\) for every \(r \geq 2\).
We also extend the definition of the map \(\eta\).
Let \(\M := \{x^{\alpha} \mid \alpha \in A\}\) be the set of all monomials on the variables \(\{x_{i}\}_{i \geq 1}\).
Given a Hall set \(\HHinf{}\) on \(\ff_{\infty}\), we define the map \(\eta: \HHinf{} \to \M\) inductively.
We put \(\eta(X_{i}) := x_{i}\) for all \(i \geq 1\), which defines \(\eta\) on \(\HHinf{1}\).
Now, suppose that \(\eta\) is defined on \(\bigcup_{i = 1}^{n - 1}\HHinf{i}\).
Let \(X \in \HHinf{n}\).
Then \(X = [Y, Z]\) for some \(Y \in \HHinf{k}, Z \in \HHinf{l}\) where \(k + l = n\).
We then define
\[
	\eta(X) := \eta(Y) \eta(Z).
\]

\begin{theorem}	\label{theo:LinkFkandfrk}
	Let \(k \geq 1\) be arbitrary.
Define
	\[
		F_{k} := \sum_{X \in \HHinf{k}} \eta(X).
	\]
	and
	\[
		f_{k} := \sum_{i = 0}^{\infty} (-1)^{i}e_{i}[F_{k}],
	\]
	Then, for all integers \(r \geq 2\)
	\[
		\Phi_{r}(f_{k}) = f_{r, k}.
	\]
\end{theorem}

\begin{proof}
	Fix \(r \geq 2\).
Let \(\HHinf{}\) be a Hall set on \(\ff_{\infty}\).
In the discussion following \cref{lem:eigenvaluesphii}, we argued that, if \(\HH^{r}\) is a Hall basis of \(\ff_{r}\), the set
\[
	\{\eta(X) \mid X \in \HH^{r}_{k}\}
\]
is symmetric (counted with multiplicity).
Consequently, the polynomial \(F_{r, k} := \sum_{X \in \HH^{r}_{k}} \eta(X)\) is also symmetric.
By construction of \(\HHinf{}\), \(\Phi_{r}(F_{k}) = F_{r, k}\).
Since \(F_{r, k}\) is symmetric for each \(r\), \cref{lem:CharacterisationSymmetricFunctions} implies that \(F_{k}\) is a symmetric function.

By \cref{prop:Product1-MonomialsEqualsAlternatingPlethysmElementary},
\[
	f_{k} = \prod_{X \in \HHinf{k}} (1 - \eta(X))
\]
and therefore
\[
	\Phi_{r}(f_{k}) = \Phi_{r}\left(\prod_{X \in \HHinf{k}} (1 - \eta(X))\right) = \prod_{X \in \HH^{r}_{k}} (1 - \eta(X)) = f_{r, k}
\]
by definition of \(\Phi_{r}\).
\end{proof}

Thus, combined with \cref{theo:ReidemeisterNumberSymmetricPolynomialEigenvalues}, we see that, for all \(r \geq 2\) and \(c \geq 2\), all information about Reidemeister numbers on \(N_{r, c}\) is contained in the \(f_{k}\) for \(k \in \{1, \ldots, c\}\).
We also start to see how the Reidemeister spectrum of \(N_{r, c}\) and plethysms are related.

Recall that our focus is to determine the polynomials \(g_{r, k}\), which express the polynomials \(f_{r, k}\) in terms of elementary symmetric polynomials.
By \cref{theo:SymmetricFunctionsInTermsOfElementarySymmetricFunctions}, 
\[
	f_{k} = \sum_{\lambda \in \Par} c_{\lambda, k} e_{\lambda}
\]
for some \(c_{\lambda, k} \in \Z\).
The polynomial \(g_{r, k}\) can then be derived from this equality by applying \(\Phi_{r}\) to both sides.
In \cref{sec:plethysmsBackToSpectra}, we elaborate more on this and explain the behaviour we observed in the table at the beginning of this section.

\subsection{Schur functions}
The next step in our journey is to show how Schur functions come into play in the expressions from \cref{theo:LinkFkandfrk}.
We start by formally introducing Schur functions and stating the major open problem regarding plethysms of Schur functions.
We opt for a more combinatorial way to do so, but there are many other ways.
All of them might seem unmotivated when first encountering Schur functions.
We have, however, tried to highlight their importance and/or naturality already in the introduction and will continue to do so further down this article.
For more information and motivation on Schur functions, we refer the reader to \cite[7.10 \& 7.15]{Stanley12a}.

We need to introduce four concepts before we can define Schur functions.
Let \(\lambda\) and \(\mu\) be two partitions.

\begin{itemize}
	\item The \emph{Young diagram} of \(\lambda\) is a graphical representation of \(\lambda\) consisting of a left-justified array of boxes below each other, where the \(i\)th row contains \(\lambda_{i}\) boxes.
We index the rows as in a matrix, so from top to bottom.

	For example, the Young diagram of \((4, 2, 1)\) is depicted in \cref{fig:YoungDiagram421}.
\begin{figure*}[h]
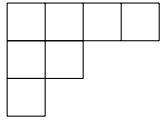

	\centering
	\drawyoungdiagram{4, 2, 1}
	\caption{Young diagram of the partition \((4, 2, 1)\)}
	\label{fig:YoungDiagram421}
\end{figure*}

	\item The \emph{monomial symmetric function} \(m_{\lambda}\) is given by
\[
	m_{\lambda} := \sum_{\substack{\alpha \in A \\ \exists \sigma \in \Sym{\Zpos}: \sigma \cdot \alpha = \lambda}} x^{\alpha}.
\]
	For example,
	\[
		m_{(4, 2, 1)} = \sum_{i \ne j \ne k \ne i} x_{i}^{4}x_{j}^{2}x_{k}.
	\]

	\item A \emph{semi-standard Young tableau} (SSYT) of type \(\lambda\) and weight \(\mu\) is a Young diagram of shape \(\lambda\) in which the number \(i\) is filled in \(\mu_{i}\) times, in such a way that the rows are non-decreasing and the columns are (strictly) increasing.
	
	For example, \cref{fig:SSYT421} depicts two SSYTs of type \((4, 2, 1)\) and weight \((3, 2, 2)\).
	
\begin{figure*}[h]
	\centering
	\begin{tikzpicture}[scale=0.5]
		\foreach \i [count=\j] in {4, 2, 1}{
			\draw (0, -\j) grid ($({\i}, {-\j + 1})$);
		}
		\node at (0.5, -0.5) {\(1\)};
		\node at (1.5, -0.5) {\(1\)};
		\node at (2.5, -0.5) {\(1\)};
		\node at (3.5, -0.5) {\(3\)};
		\node at (0.5, -1.5) {\(2\)};
		\node at (1.5, -1.5) {\(2\)};
		\node at (0.5, -2.5) {\(3\)};
	\end{tikzpicture}
	\hspace{1cm}
	\begin{tikzpicture}[scale=0.5]
		\foreach \i [count=\j] in {4, 2, 1}{
			\draw (0, -\j) grid ($({\i}, {-\j + 1})$);
		}
		\node at (0.5, -0.5) {\(1\)};
		\node at (1.5, -0.5) {\(1\)};
		\node at (2.5, -0.5) {\(1\)};
		\node at (3.5, -0.5) {\(2\)};
		\node at (0.5, -1.5) {\(2\)};
		\node at (1.5, -1.5) {\(3\)};
		\node at (0.5, -2.5) {\(3\)};
	\end{tikzpicture}
	\caption{Semi-standard Young tableaux of type \((4, 2, 1)\) and weight \((3, 2, 2)\)}
	\label{fig:SSYT421}
\end{figure*}
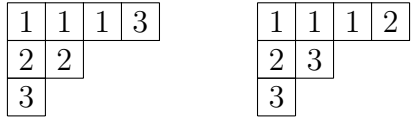
	
	\item The \emph{Kostka number \(K_{\lambda \mu}\)} is the number of SSYTs of type \(\lambda\) and weight \(\mu\).
	
	For example, for \(\lambda = (4, 2, 1)\) and \(\mu = (3, 3, 2)\), \(K_{\lambda \mu} = 2\).
\end{itemize}

Finally, the \emph{Schur function \(s_{\lambda}\)} is defined as
\[
	s_{\lambda} := \sum_{\mu \in \Par} K_{\lambda \mu} m_{\mu}.
\]
To finish the example of \(\lambda = (4, 2, 1)\), the associated Schur function is
\begin{align*}
	s_{(4, 2, 1)}	&=	35m_{1, 1, 1, 1, 1, 1, 1} + 20m_{2, 1, 1, 1, 1, 1} + 11m_{2, 2, 1, 1, 1} + 6m_{2, 2, 2, 1} + 8m_{3, 1, 1, 1, 1} 	\\
					& \quad + 4m_{3, 2, 1, 1} + 2m_{3, 2, 2} + m_{3, 3, 1} + 2m_{4, 1, 1, 1} + m_{4, 2, 1}.
\end{align*}

\begin{theorem}
	Each symmetric function \(f \in \Lambda\) can be written as a series of Schur functions, \ie there exist \(c_{\lambda}	\in \Z\) with \(\lambda \in \Par\) such that
	\[
		f = \sum_{\lambda \in \Par} c_{\lambda} s_{\lambda}.
	\]
\end{theorem}
\begin{proof}
	We proceed in a similar fashion as we did for the elementary symmetric functions.
Write \(f = \sum_{n \geq 1} f_{n}\) in its homogeneous parts.
In \cite[7.10.6~Corollary]{Stanley12a}, it is proven that each \(f_{n}\) can be written as
	\[
		f_{n} = \sum_{\lambda \in \Par(n)} c_{\lambda} s_{\lambda}
	\]
	for some \(c_{\lambda} \in \Z\).
By combining all of these expressions, we obtain the series for \(f\).
\end{proof}

For instance, \(e_{k} = s_{1^{k}}\) for all \(k \geq 1\), where \(1^{k}\) is a short-hand notation for the partition consisting of \(k\) ones.
This equality can be verified using the definition of both symmetric functions.
We refer to the sum \(\sum_{\lambda \in \Par} c_{\lambda} s_{\lambda}\) as the \emph{Schur expansion} of \(f\).

If all the \(c_{\lambda}\) are non-negative, we say that \(f\) is \emph{Schur positive}.
The following then is a classical result:
\begin{theorem}[{\cite[A.2.7~ Theorem]{Stanley12a}}] \label{theo:PlethysmOfSchurPositiveIsSchurPositive}
	Let \(f, g\) be two Schur positive symmetric functions.
Then \(f[g]\) is again Schur positive.
\end{theorem}

Thus, if \(f\) and \(g\) are Schur positive, then the \(c_{\lambda}\) in the Schur expansion of \(f[g]\) are non-negative integers.
Hence, they might admit a combinatorial interpretation.
For the case where \(f\) and \(g\) are Schur functions themselves, we get one of the major open problems in combinatorics:

\begin{objective}	\label{obj:PlethysmSchurFunctions}
	Let \(\lambda, \mu\) be two partitions.
Determine the coefficients \(a_{\lambda, \mu}^{\nu}\) in the Schur expansion
	\[
		s_{\lambda}[s_{\mu}] = \sum_{\nu} a_{\lambda, \mu}^{\nu} s_{\nu}.
	\]
\end{objective}

Ultimately, we want to relate Reidemeister numbers on free nilpotent groups to this open problem.
In \cref{theo:LinkFkandfrk}, we showed that \(\Phi_{r}(f_{k}) = f_{r, k}\) for all \(r \geq 2\) and \(k \geq 1\), where
\[
	F_{k} = \sum_{X \in \HHinf{k}} \eta(X)
\]
and
\[
	f_{k} = \sum_{i = 0}^{\infty} (-1)^{i}e_{i}[F_{k}].
\]
Since \(e_{i} = s_{1^{i}}\) for all \(i \geq 1\), we can write 
\[
	f_{k} = \sum_{i = 0}^{\infty} (-1)^{i}s_{1^{i}}[F_{k}].
\]
We now argue that \(F_{k}\) is in fact Schur positive and, in some cases, even a Schur function.

\subsection{Schur positivity of \(F_{n}\)}	\label{subsec:SchurFunctionRepresentationTheory}

Symmetric functions, and especially Schur functions, naturally occur in representation theory, in particular in the representation theory of \(\Sym{n}\) and of \(\GL(n, \C)\).
We provide a high level overview of both and prove a relation between them in the context of free Lie algebras.
This relation will yield a proof that the symmetric functions \(F_{n}\) are Schur positive.

The discussion of representations of \(\GL(n, \C)\) below is based on \cite[Chapter 7, Appendix 2]{Stanley12a}.

\begin{defin}
	Let \(n, m \geq 1\) be two integers.
A representation \(\rho: \GL(n, \C) \to \GL(m, \C)\) is called \emph{polynomial} if the entries of \(\rho(A)\) are polynomials in the entries of \(A \in \GL(n, \C)\).
\end{defin}

An example of a polynomial representation is the determinant map\[
\det: \GL(n, \C) \to \GL(1, \C): A \mapsto \det(A).
\]
If \(\rho: \GL(n, \C) \to \GL(m, \C)\) is a polynomial representation, then there exists a multiset \(\M_{\rho}\) consisting of monomials in \(n\) variables such that, if \(A \in \GL(n, \C)\) has eigenvalues \(\lambda_{1}, \ldots, \lambda_{n}\), \(\rho(A)\) has eigenvalues \(\lambda^{\alpha}\), where \(x^{\alpha} \in \M_{\rho}\).
We then define the \emph{character} of \(\rho\) to be the polynomial
\[
	\Char(\rho) := \sum_{x^{\alpha} \in \M_{\rho}} x^{\alpha}.
\]

From the definition, it readily follows that the character of a sum of polynomial representations is the sum of the characters.
Furthermore, the character of a polynomial representation is always a \emph{symmetric} polynomial in \(x_{1}, \ldots, x_{n}\): for \(A \in \GL(n, \C)\), any permutation \(\lambda_{\sigma(1)}, \ldots, \lambda_{\sigma(n)}\) of the eigenvalues of \(A\) must yield the same eigenvalues of \(\rho(A)\).

The relation between Schur functions and representation theory of \(\GL(n, \C)\) is captured in the following theorem.
For a partition \(\lambda = (\lambda_{1}, \ldots, \lambda_{k} , \ldots)\), we define the \emph{length} to be the maximal \(i\) such that \(\lambda_{i} \ne 0\).

\begin{theorem}[{\cite[A2.4 Theorem(IV)]{Stanley12a}}]
	Let \(n \geq 1\) be an integer.
	\begin{enumerate}
		\item The irreducible polynomial representations of \(\GL(n, \C)\) are in one-to-one correspondence with the partitions of length at most \(n\).
		\item If the irreducible polynomial representation \(\rho\) corresponds to \(\lambda\) under this correspondence, then
		\[
			\Char(\rho) = s_{\lambda}(x_{1}, \ldots, x_{n}).
		\]
	\end{enumerate}
\end{theorem}

Now, we move to representations of the symmetric group.
We first need yet another family of symmetric functions, the power symmetric functions.
\begin{defin}
	Let \(n \geq 1\) and let \(\lambda \in \Par\) be a partition.
We define the \emph{power symmetric functions} as
	\[
		p_{n} := \sum_{i \geq 1} x_{i}^{n}
	\]
	and
	\[
		p_{\lambda} := \prod_{i \geq 1} p_{\lambda_{i}}.
	\]
\end{defin}

Let \(n \geq 1\) be an integer, \(\rho: \Sym{n} \to \GL(V)\) a representation of \(\Sym{n}\) and \(\chi_{\rho}\) its (ordinary) character.
The \emph{Frobenius character of \(\rho\)} is the symmetric function
\[
	\ch(\rho) := \frac{1}{n!} \sum_{\sigma \in \Sym{n}} \chi_{\rho}(\sigma) p_{c(\sigma)}
\]
where \(c(\sigma) = (c_{1}, c_{2}, \ldots, c_{k}, \ldots)\) is the cycle type of \(\sigma\), \ie the cycle decomposition of \(\sigma\) consists of cycles of length \(c_{1} \geq c_{2} \geq \ldots \geq c_{k} \geq \ldots\).

\begin{remark}
	For clarity, we list all the uses of the term \emph{character} and their corresponding notations:
	\begin{itemize}
		\item The \emph{character} of a polynomial representation \(\rho\) of \(\GL(n, \C)\) is \(\Char(\rho)\).
		\item The \emph{ordinary character} of a representation \(\rho\) of \(\Sym{n}\) is \(\chi_{\rho}\).
		\item The \emph{Frobenius character} of a representation \(\rho\) of \(\Sym{n}\) is \(\ch(\rho)\).
	\end{itemize}
\end{remark}

The irreducible representations of \(\Sym{n}\) are indexed by the partitions of \(n\), see \cite[\S4.2]{FultonHarris04}.
This indexation behaves nicely under the Frobenius character.

\begin{theorem}[{\cite[(7.6)(i), p.~114]{MacDonald95}}]
	Let \(n \geq 1\).
For each partition \(\lambda \vdash n\), the Frobenius character of the irreducible representation \(\rho_{\lambda}\) of \(\Sym{n}\) equals \(\ch(\rho_{\lambda}) = s_{\lambda}\).
\end{theorem}

The Frobenius character of a sum of representations is the sum of their respective Frobenius characters.
Hence, the Frobenius character of \emph{any} representation of \(\Sym{n}\) is a non-negative integral linear combination of Schur functions.
In other words, it is a Schur positive symmetric function.

To prove that \(F_{n}\) is Schur positive for each \(n \geq 1\), we construct a representation of \(\Sym{n}\) whose Frobenius character equals \(F_{n}\).

Let \(N \geq 1\) and let \(\ff_{N}\) be the free Lie algebra on \(X_{1}, \ldots, X_{N}\).
We construct a set \(\B(N) := \bigcup_{i = 1}^{\infty} \B_{i}(N)\) similar to a Hall basis of \(\ff_{N}\):
\begin{itemize}
	\item Put \(\B_{1}(N) := \{X_{1}, \ldots, X_{N}\}\).
	\item For \(i \geq 2\), suppose \(\B_{1}(N)\) up to \(\B_{i - 1}(N)\) have been constructed.
	Define \(\B_{i}(N)\) to be the set of all Lie brackets \([Y, Z]\) with \(Y \in \B_{k}(N)\) and \(Z \in \B_{l}(N)\) such that \(k + l = i\).
\end{itemize}	
In other words, \(\B_{i}(N)\) contains all possible Lie brackets of \(i\) elements from \(\B_{1}(N)\), where the brackets are taken in any order.
For instance, \(\B_{4}(N)\) contains both \([[[X_{1}, X_{2}], X_{3}], X_{4}]\) and \([[X_{1}, X_{2}], [X_{1}, X_{3}]]\) (if \(N \geq 4\)).

Note that \(\B_{i}(N)\) is not a linearly independent set for \(i \geq 2\), as \(\B_{2}(N)\) contains both \([X_{1}, X_{2}]\) and \([X_{2}, X_{1}]\), for example.
It is therefore distinct from the subset \(\HH_{i}\) of a Hall basis \(\HH\) of \(\ff_{N}\) (for \(i \geq 2\)).
The inclusion \(\HH_{i} \subseteq \B_{i}(N)\), however, does hold for all \(i \geq 1\).

Now, for an integer \(n \geq 1\), we define \(\ffsupport{n}\) to be the \emph{vector subspace} of \(\ff_{n}\) spanned by all elements in \(\B_{n}(n)\) that contain each \(X_{i}\) exactly once.
For example, \(\ffsupport{3}\) is spanned by the vectors
\[
	\{[[X_{i}, X_{j}], X_{k}] \mid \{i, j, k\} = \{1, 2, 3\}\} \cup \{[X_{i}, [X_{j}, X_{k}]] \mid \{i, j, k\} = \{1, 2, 3\}\}.
\]

The group \(\Sym{n}\) acts on \(\ff_{n}\) in a natural way by permuting the \(X_{i}\), that is, for \(\sigma \in \Sym{n}\), we put \(\sigma \cdot X_{i} := X_{\sigma(i)}\) for all \(i \in \{1, \ldots, n\}\).
As this action maps \(\B_{n}(n)\) onto itself, we get a representation \(\rho^{(n)}: \Sym{n} \to \GL(\ffsupport{n})\).

\begin{theorem}	\label{theo:FnFrobeniusCharacter}
	Let \(n \geq 1\) be an integer.
The symmetric function \(F_{n}\) is the Frobenius character of the representation \(\rho^{(n)}: \Sym{n} \to \GL(\ffsupport{n})\).
\end{theorem}

\begin{proof}
	A.\;Garsia proves \cite[Theorem~5.2~\&~Theorem~5.3]{Garsia90} that the Frobenius character of \(\rho^{(n)}\) is given by
	\begin{equation}	\label{eq:FrobeniusCharacterRhon}
		\ch(\rho^{(n)})	=	\frac{1}{n} \sum_{d \mid n} \mu(d) p_{d}^{n / d},
	\end{equation}
	where \(\mu\) is the Möbius function, as defined in \cref{theo:DimensionHallBasis}.
	Strictly speaking, in the proof of \cite[Theorem~5.2]{Garsia90}, he works with symmetric \emph{polynomials} rather than symmetric functions in the definition of the Frobenius character and the power symmetric polynomials.
However, his result and proof are still valid when we work with symmetric functions instead.

	Garsia additionally proves that a particular representation of \(\GL(N, \C)\) for certain values of \(N\) essentially has the same character as above.
We can express this character in a different way using the results of \cref{sec:FromReidemeisterToPolynomials}, which we will then use to prove the theorem.
		
	We start by constructing this representation.
Let \(N \geq n\) be an integer.
Recall the set \(\B(N) = \bigcup_{i = 1}^{\infty} \B_{i}(N)\) of Lie brackets in the free Lie algebra \(\ff_{N}\).
We define, for every \(i \geq 1\), \(\fflength{N}{i}:= \Span(\B_{i}(N))\).
Note that, if \(\HH = \bigcup_{i = 1}^{\infty} \HH_{i}\) is a Hall basis of \(\ff_{N}\), then \(\HH_{i}\) is a basis of \(\fflength{N}{i}\) for \(i \geq 1\).
Indeed, since \(\HH_{i} \subseteq \B_{i}(N)\) for all \( i \geq 1\), \(\Span(\HH_{i}) \subseteq \fflength{N}{i}\).
Furthermore, \(\ff_{N} = \bigoplus_{i = 1}^{\infty} \fflength{N}{i}\), so as \(\HH\) is a basis of \(\ff_{N}\), each \(\HH_{i}\) must span the whole of \(\fflength{N}{i}\).
Being part of a basis, \(\HH_{i}\) is also linearly independent.

Let \(A \in \GL(N, \C)\) be arbitrary.
Since \(\ff_{N}\) is the free Lie algebra on \(X_{1}, \ldots, X_{N}\), the matrix \(A\) induces a Lie algebra automorphism \(\phi_{A}\) of \(\ff_{N}\) by defining \(\phi_{A}(X_{i}) := \sum_{j = 1}^{N} A_{ij} X_{j}\).
By linearity of the Lie bracket, the image of an element in \(\B_{n}(N)\) is a linear combination of elements of \(\B_{n}(N)\).
Hence, \(\phi_{A}\) restricts to an automorphism \(\restr{\phi_{A}}{\fflength{N}{n}}\) of \(\fflength{N}{n}\).
Thus, we get a map
\[
	\rho_{N, n}: \GL(N, \C) \to \GL(\fflength{N}{n}): A \mapsto \restr{\phi_{A}}{\fflength{N}{n}},
\]
which in addition is a group homomorphism.
In other words, \(\rho_{N, n}\) is a representation of \(\GL(N, \C)\).
Note that by the linearity of the Lie bracket, \(\rho_{N, n}\) is also a polynomial representation.
	
	Now, Garsia proves \cite[Theorem~6.4]{Garsia90} that, for all \(N \geq n\), the character of \(\rho_{N, n}\) is equal to
	\[
		\Char(\rho_{N, n}) = \frac{1}{n} \sum_{d \mid n} \mu(d) p_{d}^{n / d}(x_{1}, \ldots, x_{N}) = \Phi_{N}\left(\frac{1}{n} \sum_{d \mid n} \mu(d) p_{d}^{n / d}\right)
	\]
	
	On the other hand, the representation \(\rho_{N, n}\) is closely related to \cref{lem:eigenvaluesphii}.
There, we started from a Lie algebra automorphism \(\phi_{*}\) of \(\nrc\) and looked at the induced maps on \(\gamma_{i}(\nrc) / \gamma_{i + 1}(\nrc)\).
Now, let \(c \geq n\).
Since \(\gamma_{c + 1}(\ff_{N})\) is characteristic, \(\phi_{A}\) as constructed above induces an automorphism \(\phi_{A, *}\) of \(\nn_{N, c}\).
We show that the eigenvalues of \(\phi_{A, *, n} \in \GL(\gamma_{n}(\nn_{N, c}) / \gamma_{n + 1}(\nn_{N, c}))\) and \(\restr{\phi_{A}}{\fflength{N}{n}}\) are equal.

To that end, let \(\HH = \bigcup_{i = 1}^{\infty} \HH_{i}\) be a Hall basis of \(\ff_{N}\).
As argued earlier, the subset \(\HH_{n}\) is a basis of \(\fflength{N}{n}\), and by \cref{lem:basisFactorsLCSnrc}, the natural projections of the elements in \(\HH_{n}\) also form a basis of \(\gamma_{n}(\nn_{N, c}) / \gamma_{n + 1}(\nn_{N, c})\).
Hence, as vector spaces, \(\fflength{N}{n}\) and \(\gamma_{n}(\nn_{N, c}) / \gamma_{n + 1}(\nn_{N, c})\) are isomorphic.
The map that sends every element of \(\HH_{n}\) to its natural projection gives rise to an isomorphism \(\psi: \fflength{N}{n} \to \gamma_{n}(\nn_{N, c}) / \gamma_{n + 1}(\nn_{N, c})\).
This isomorphism satisfies \(\inv{\psi} \circ \phi_{A, *, n} \circ \psi = \restr{\phi_{A}}{\fflength{N}{n}}\).
Hence, \(\phi_{A, *, n} \in \GL(\gamma_{n}(\nn_{N, c}) / \gamma_{n + 1}(\nn_{N, c}))\) and \(\restr{\phi_{A}}{\fflength{N}{n}}\) have the same eigenvalues.

	Subsequently, it follows from \cref{lem:eigenvaluesphii} that the character of \(\rho_{N, n}\) is also equal to
	\[
		\Char(\rho_{N, n}) = \sum_{X \in \HH_{n}} \eta(X)
	\]
	where \(\HH\) is a Hall basis of \(\ff_{N}\) and \(\eta: \HH \to \polZ{N}\) is the map that associates to an element of \(\HH\) its corresponding monomial.
In other words, the character of \(\rho_{N, n}\) is equal to \(\Phi_{N}(F_{n})\).
Thus, combined with \eqref{eq:FrobeniusCharacterRhon}, we have the equality
	\[
		\Phi_{N}(F_{n}) = \Phi_{N}(\ch(\rho^{(n)}))
	\]
	for all \(N \geq n\), as \(N\) was arbitrary.
By \cref{lem:equalityConditionPowerSeries}, it follows that \(F_{n} = \ch(\rho^{(n)})\).
\end{proof}

\begin{cor}
	For each \(n \geq 1\) and \(i \geq 1\), both \(F_{n}\) and \(e_{i}[F_{n}]\) are Schur positive.
\end{cor}
\begin{proof}
	By the previous theorem, \(F_{n}\) is the Frobenius character of a representation of \(\Sym{n}\).
Consequently, it is Schur positive.
\cref{theo:PlethysmOfSchurPositiveIsSchurPositive} then implies that \(e_{i}[F_{n}]\) is Schur positive as well for all \(i \geq 1\).
\end{proof}

At this point, our journey is almost at an end.
We have established a link between the Reidemeister spectra of free nilpotent groups and the plethysm of Schur positive functions.
To show that \(F_{n}\) is sometimes a true Schur function, not merely Schur positive, we end this section by stating a remarkable theorem by W.\;Kraskiewicz and J.\;Weyman.
They give an explicit and elegant combinatorial interpretation of the coefficients in the Schur decomposition of \(\ch(\rho^{(n)})\).
As we have established that \(F_{n}\) is equal to \(\ch(\rho^{(n)})\), we obtain a combinatorial interpretation of the coefficients in the Schur decomposition of \(F_{n}\).

Let \(\lambda \vdash n\) be a partition.
A \emph{standard Young tableau of type \(\lambda\)} is a SSYT of type \(\lambda\) and weight \((1^{n})\).
In other words, all the entries in the Young tableau have to be distinct.
An entry \(i\) in a standard Young tableau is called a \emph{descent} if \(i + 1\) occurs in some row below the entry \(i\).
The \emph{major index} of a standard Young tableau is defined as the sum of all descents.
For example, the descents of the standard Young tableau in \cref{fig:DescentsYT} are \(3\), \(5\) and \(6\), and its major index is \(14\).

\begin{figure}[h]
	\centering
	\begin{tikzpicture}[scale=0.5]
		\foreach \i [count=\j] in {4, 2, 1}{
			\draw (0, -\j) grid ($({\i}, {-\j + 1})$);
		}
		\node at (0.5, -0.5) {\(1\)};
		\node at (1.5, -0.5) {\(2\)};
		\node at (2.5, -0.5) {\(3\)};
		\node at (3.5, -0.5) {\(5\)};
		\node at (0.5, -1.5) {\(4\)};
		\node at (1.5, -1.5) {\(6\)};
		\node at (0.5, -2.5) {\(7\)};
	\end{tikzpicture}
	\caption{Standard Young tableau of shape \((4, 2, 1)\)}
	\label{fig:DescentsYT}
\end{figure}

\begin{theorem}[{\cite{KraskiewiczWeyman01}}]
	Let \(n \geq 1\) be an integer and \(\lambda\) a partition of \(n\).
The coefficient of \(s_{\lambda}\) in \(F_{n}\) is equal to the number of standard Young tableaux of shape \(\lambda\) with major index congruent to \(1\) modulo \(n\).
\end{theorem}

For example, \(F_{1} = s_{1}\), \(F_{2} = s_{1, 1}\), \(F_{3} = s_{2, 1}\) and \(F_{4} = s_{3, 1} + s_{2, 1, 1}\).
Thus, in turn,
\[
	f_{1} = \sum_{i = 0}^{\infty} (-1)^{i} s_{1^{i}}[s_{1}] = \sum_{i = 0}^{\infty} (-1)^{i} e_{i}[e_{1}] = \sum_{i = 0}^{\infty} (-1)^{i} e_{i}
\]
and
\[
	f_{2} = \sum_{i = 0}^{\infty} (-1)^{i} s_{1^{i}}[s_{1, 1}].
\]

\section{From plethysms of Schur functions back to Reidemeister spectra}	\label{sec:plethysmsBackToSpectra}
Our journey is coming to an end.
We have established a connection between two open problems: fully determining the Reidemeister spectra of free nilpotent groups and determining the coefficients in the Schur expansion of the plethysms of Schur functions.
Since the latter is a major open problem in combinatorics, this connection in some sense indicates why the computations of Dekimpe, Tertooy and Vargas are tedious and, at first sight, hard to generalise.

We want to elaborate further on two aspects.
Firstly, we want to delve deeper into the `practical' aspect of this connection: how does progress in the realm of symmetric functions translate into progress in the realm of Reidemeister spectra? Secondly, it might feel a bit unsatisfying to speak of a connection between both open problems if most of the \(F_{n}\) are not Schur functions.
Therefore, we want to discuss the problem of determining the Reidemeister spectra of free nilpotent \emph{metabelian} groups and show that this is related to the plethysm of proper Schur functions.

\subsection{Main objectives}
At the end of \cref{subsec:FromSymmetricPolynomialsToSymmetricFunctions}, we explained that we can derive the polynomials \(g_{r, k}\) from \(f_{k}\): writing
\[
	f_{k} = \sum_{\lambda \in \Par} c_{\lambda, k} e_{\lambda}
\]
for some \(c_{\lambda, k} \in \Z\) as in \cref{theo:SymmetricFunctionsInTermsOfElementarySymmetricFunctions}, we can derive \(g_{r, k}\) from this equality by applying \(\Phi_{r}\) to both sides.
Here, we elaborate on the details.

Fix integers \(r, c, i\) and \(k\) such that \(r, c \geq 2\), \(i \geq 1\) and \(1 \leq k \leq c\).
We know from \cref{theo:LinkFkandfrk} that
\[
	f_{r, k} = \Phi_{r}(f_{k}) = \Phi_{r} \left(\sum_{i = 1}^{\infty} (-1)^{i} e_{i}[F_{k}]\right).
\]
Note that \(e_{i}\) is homogeneous of degree \(i\) and \(F_{k}\) of degree \(k\).
Hence, \(e_{i}[F_{k}]\) is homogeneous of degree \(ik\).
This implies that there are no cancellations possible among the different terms of the summation.
On the other hand, since \(f_{r, k}\) is a polynomial, only finitely many terms can survive under \(\Phi_{r}\).
To determine which, we look more closely at how \(e_{i}[F_{k}]\) is computed.
Write \(F_{k} = \sum_{j \geq 1} x^{\alpha_{j}}\) for some \(\alpha_{j} \in A\).

For \(e_{i}[F_{k}]\) to (partially) survive under \(\Phi_{r}\), at least one monomial must contain only the variables \(x_{1}\) up to \(x_{r}\).
Suppose that \(i > N(r, k)\), the size of \(\HH_{k}\), where \(\HH\) is a Hall basis of \(\ff_{r}\).
Let \(x_{j_{1}} \cdots x_{j_{i}}\) be a monomial occurring in \(e_{i}\) for some positive integers \(j_{1} < \ldots < j_{i}\).
To compute the plethysm \(e_{i}[F_{k}]\), we have to substitute \(x_{j_{l}} = x^{\alpha_{j_{l}}}\) for each \(l \in \{1, \ldots, i\}\).
All \(j_{l}\) are distinct and \(i > N(r, k)\).
As there are, by \cref{theo:DimensionHallBasis}, only \(N(r, k)\) many monomials occurring in \(F_{k}\) that only contain the variables \(x_{1}\) up to \(x_{r}\), at least one of the substitutions \(x_{j_{l}} = x^{\alpha_{j_{l}}}\) involves a monomial containing \(x_{m}\) with \(m > r\).
The resulting monomial in \(e_{i}[F_{k}]\) then vanishes under \(\Phi_{r}\).
As this holds for all monomials occurring in \(e_{i}\), \(e_{i}[F_{k}]\) vanishes completely under \(\Phi_{r}\).

Consequently,
\begin{equation}	\label{eq:frkUpToNecklacePolynomial}
	f_{r, k} = \Phi_{r}\left(\sum_{i = 1}^{N(r, k)} (-1)^{i} e_{i}[F_{k}]\right)
\end{equation}

In other words, if we want to compute \(\SpecR(N_{r, c})\) for a fixed value of \(c\) and for \(r\) up to a certain bound, say \(R\), we only have to effectively compute \(e_{i}[F_k]\) for \(i \leq N(R, k)\) and \(k \leq c\).

Next, suppose we have an expansion of \(e_{i}[F_{k}]\) in terms of the elementary symmetric functions, say, \(e_{i}[F_{k}] = \sum_{\lambda \in \Par(ik)} c_{\lambda, i, k} e_{\lambda}\), where \(c_{\lambda, i, k} \in \Z\) for all \( i \geq 1, k \geq 1\) and \(\lambda \in \Par(ik)\).
We know that only partitions of \(ik\) occur, as \(e_{i}[F_{k}]\) is homogeneous of degree \(ik\), and so is \(e_{\lambda}\) for \(\lambda \in \Par(ik)\).
From \eqref{eq:frkUpToNecklacePolynomial}, we obtain
\[
	f_{r, k} = \Phi_{r}\left(\sum_{i = 1}^{N(r, k)} (-1)^{i} e_{i}[F_{k}]\right) = \Phi_{r}\left(\sum_{i = 1}^{N(r, k)} (-1)^{i} \sum_{\lambda \in \Par(ik)} c_{\lambda, i, k} e_{\lambda}\right)
\]

Since \(\Phi_{r}(e_{s}) = 0\) for all \(s > r\), the only \(e_{\lambda}\) that survive under \(\Phi_{r}\) are those for which \(\lambda_{i} \leq r\) for all \(i\), or, equivalently, those for which \(\lambda_{1} \leq r\), as a partition is a non-increasing sequence.
Thus,
\[
	f_{r, k} = \sum_{i = 1}^{N(r, k)} (-1)^{i} \sum_{\substack{\lambda \in \Par(ik) \\ \lambda_{1} \leq r}} c_{\lambda, i, k} e_{\lambda}(x_{1}, \ldots, x_{r}).
\]
From this expression, we obtain the polynomial \(g_{r, k}\) from \cref{theo:ReidemeisterNumbersInTermsOfCoefficientCharPol}.

To compute \(\SpecR(N_{r, c})\), we then have to compute all possible values of
\[
	\inftynormb{\prod_{k = 1}^{c} \sum_{i = 1}^{N(r, k)} (-1)^{i} \sum_{\substack{\lambda \in \Par(ik) \\ \lambda_{1} \leq r}} c_{\lambda, i, k} y_{\lambda}}
\]
by letting \(y_{1}, \ldots, y_{r}\) run over all integers, with the restriction that \(y_{r} \in \{-1, 1\}\) and where 
\[
	y_{\lambda} := \prod_{i = 1}^{\infty} y_{\lambda_{i}}.
\]

Also note that
\begin{align*}
	f_{r, k}	&= \sum_{i = 1}^{N(r, k)} (-1)^{i} \sum_{\substack{\lambda \in \Par(ik) \\ \lambda_{1} \leq r}} c_{\lambda, i, k} e_{\lambda}(x_{1}, \ldots, x_{r})	\\
			&= \sum_{i = 1}^{N(r, k)} (-1)^{i} \sum_{\substack{\lambda \in \Par(ik) \\ \lambda_{1} = r}} c_{\lambda, i, k} e_{\lambda}(x_{1}, \ldots, x_{r}) + \sum_{i = 1}^{N(r, k)} (-1)^{i} \sum_{\substack{\lambda \in \Par(ik) \\ \lambda_{1} \leq r - 1}} c_{\lambda, i, k} e_{\lambda}(x_{1}, \ldots, x_{r})	\\
			&= \sum_{i = 1}^{N(r, k)} (-1)^{i} \sum_{\substack{\lambda \in \Par(ik) \\ \lambda_{1} = r}} c_{\lambda, i, k} e_{\lambda}(x_{1}, \ldots, x_{r}) + f_{r - 1, k}.
\end{align*}
Each summand in the first term has a factor \(e_{r}\), as \(\lambda_{1} = r\) for all partitions in the summation.
Therefore, setting \(e_{r} = 0\) in the expression for \(f_{r, k}\) yields the expression for \(f_{r - 1, k}\).
Put differently, the polynomial \(g_{r - 1, k}\) can be derived from the polynomial \(g_{r, k}\) by putting \(e_{r} = 0\).
This rigorously explains the behaviour we observed in the table in \cref{sec:Combinatorics}.

Summarised, from a group theoretical point of view with focus on Reidemeister spectra, we obtain two objectives:
\begin{objective}	\label{obj:GroupTheoreticalObjectiveNrc}
	For all \(i, k \geq 1\), determine the expansion of \(e_{i}[F_{k}]\) in terms of the elementary symmetric functions.
\end{objective}

\begin{objective}	\label{obj:DiophantineObjectiveNrc}
	Suppose we know the expansion of \(e_{i}[F_{k}]\) in terms of elementary symmetric functions for all \(i, k \geq 1\), say, \(e_{i}[F_{k}] = \sum_{\lambda \in \Par(ik)} c_{\lambda, i, k} e_{\lambda}\) for some \(c_{\lambda, i, k} \in \Z\).
	Given integers \(r, c \geq 2\) and a positive integer \(n \geq 1\), determine whether the equation
	\[
		\pm n = \prod_{k = 1}^{c} \sum_{i = 1}^{N(r, k)} (-1)^{i} \sum_{\substack{\lambda \in \Par(ik) \\ \lambda_{1} \leq r}} c_{\lambda, i, k} y_{\lambda}
	\]
	has a solution with \(y_{1}, \ldots, y_{r - 1} \in \Z\) and \(y_{r} \in \{-1, 1\}\).
\end{objective}

As mentioned in the introduction, the Reidemeister spectrum of \(N_{r, 2}\) has already been determined for every \(r \geq 2\), see \cite{DekimpeTertooyVargas20,Romankov11}.
However, the proof by Dekimpe, Tertooy and Vargas is rather ad hoc, so an expression for the expansion in terms of elementary symmetric functions might give an alternative proof that is generalisable.

We can also try to answer more specific questions about \(\SpecR(N_{r, c})\), besides (or as alternative to) completely determining it.
For example, for which values of \(r\) and \(c\) is \(1\) contained in \(\SpecR(N_{r, c})\)?
To answer such questions, it might be of interest to look for factorisations of \(f_{k}\) using elementary symmetric functions.
For example, recall that
\[
	f_{3} = \prod_{i \ne j} (1 - x_{i}^{2}x_{j}) \cdot \left(\prod_{i < j < k} (1 - x_{i}x_{j}x_{k})\right)^{2}.
\]
Both factors are symmetric functions, so both admit an expansion in terms of elementary symmetric functions.
Write \(f_{3} = f_{3, i \ne j} f_{3, i<j<k}^{2}\), where
\[
	f_{3, i \ne j} = \prod_{i \ne j} (1 - x_{i}^{2}x_{j}) = \sum_{n = 1}^\infty (-1)^{n} e_{n} [m_{2, 1}]
\]
and
\[
	f_{3, i < j < k} = \prod_{i < j < k} (1 - x_{i}x_{j} x_{k}) = \sum_{n = 1}^\infty (-1)^{n} e_{n}[e_{3}]
\]
by \cref{prop:Product1-MonomialsEqualsAlternatingPlethysmElementary}.
For \(1\) to be in \(\SpecR(N_{r, c})\), both \(\Phi_{r}(f_{3, i \ne j})\) and \(\Phi_{r}(f_{3, i < j < k})\) have to evaluate to \(\pm 1\) for certain values of \(e_{i}\).

From a combinatorial point of view with focus on the open problem concerning plethysms of Schur functions, the main objective is the following:
\begin{objective}	\label{obj:CombinatorialObjectiveNrc}
	For all \(i, k \geq 1\), determine the expansion of \(e_{i}[F_{k}]\) in terms of the Schur functions and find a combinatorial interpretation of the coefficients.
\end{objective}

The plethysms for \(i = 1\) are easy, since \(e_{1}[F_{k}] = F_{k}\) for all \(k \geq 1\).
Similarly, since \(F_{1} = e_{1}\), we have \(e_{i}[F_{1}] = e_{i}\) for all \(i \geq 1\).
Thus, the difficulty arises when \(i, k \geq 2\).
To determine \(\SpecR(N_{r, c})\), we have to compute \(e_{i}[F_{k}]\) for \(i \geq 1\) and \(k \in \{1, \ldots, c\}\).
For \(c = 2\), this has already been done: D.\;E.\;Littlewood has determined the Schur expansion of \(s_{1^{i}}[s_{1, 1}] = e_{i}[F_{2}]\) \cite[Equation~11.9; 1]{Littlewood50}; see also \cite[Theorem~5.7]{ColmenarejoOrellanaSaliolaSchillingZabrocki22}, \cite[Theorem~5.2]{KlivansReiner08}.
However, the expansion of \(s_{1^{i}}[s_{1, 1}]\) in elementary symmetric functions is not known.

The combinatorial and group theoretical objective are closely related.
The dual Jacobi-Trudi identity \cite[Corollary~7.16.2]{Stanley12a} expresses the Schur functions as a determinant of a matrix  with elementary symmetric functions as entries.
Therefore, we can derive the expansion of \(e_{i}[F_{k}]\) in terms of elementary symmetric functions from its Schur expansion.
However, this is quite cumbersome and perhaps there are alternatives that yield more insight and/or are more efficient.

\subsection{Reidemeister spectrum of free nilpotent metabelian groups}
The objective of determining \(\SpecR(N_{r, c})\) for all \(r, c \geq 2\) is linked to plethysms of Schur functions only when \(k \leq 3\).
However, there is a second family of groups that is closely related to the free nilpotent groups and that yields a stronger connection with plethysms of Schur functions: the free \(c\)-step nilpotent and metabelian groups of finite rank.

For a group \(G\), we define \(G^{(0)} := G\) and \(G^{(i)} := [G^{(i - 1)}, G^{(i - 1)}]\) for \(i \geq 1\).
\begin{defin}
	Let \(r, c \geq 2\) be integers.
The \emph{free \(c\)-step nilpotent and metabelian group of rank \(r\)} is the group
	\[
		M_{r, c} := \frac{F(r)}{\gamma_{c + 1}(F(r)) F(r)^{(2)}}.
	\]
\end{defin}
K.\;Dekimpe and D.\;Gon\c{c}alves have proven that \(M_{r, c}\) has the \(\Rinf\)-property if and only if \(c \geq 2r\), which is the same condition as for \(N_{r, c}\) \cite[Theorem~2.8]{DekimpeGoncalves14}.
We can do an analogous reasoning as we did for \(N_{r, c}\) to relate Reidemeister numbers on \(M_{r, c}\) to symmetric functions and plethysms of Schur functions.
We provide the results for the metabelian case, accompanied with references to the analogous results for \(N_{r, c}\).
However, we only provide the proofs that differ substantially from those for \(N_{r, c}\).

For a Lie algebra \(\g\) and \(i \geq 0\), we define \(\g^{(i)}\) in a similar fashion as we did for groups.
For integers \(r, c \geq 2\), we define the \emph{free \(c\)-step nilpotent and metabelian Lie algebra of rank \(r\)} as the Lie algebra
	\[
		\mrc := \frac{\ff_{r}}{\gamma_{c + 1}(\ff_{r}) +  \ff_{r}^{(2)}}.
	\]
Finally, for elements \(v_{1}, \ldots, v_{n}\) of a Lie algebra \(\g\), we define
\[
	[v_{1}, \ldots, v_{n}] := [[ \ldots [[v_{1}, v_{2}], v_{3}], \ldots], v_{n}].
\]

\begin{theorem}[{\cite[Theorem~2]{Chen51}, \cf \cref{def:HallBasisFreeLieAlgebra}}]	\label{theo:ChenBasis}
	Let \(r, c \geq 2\) be integers and let \(\mrc\) be the free \(c\)-step nilpotent and metabelian Lie algebra on \(X_{1}, \ldots, X_{r}\).
Let \(\CC := \bigcup_{k = 1}^{c} \CC_{k}\) be given by
	\[
		\CC_{1} := \{X_{1}, \ldots, X_{r}\}
	\]
	and
	\[
		\CC_{k} := \{[X_{i_{1}}, X_{i_{2}}, \ldots, X_{i_{k}}] \mid i_{1} > i_{2} \leq i_{3} \leq \ldots \leq i_{k}, \forall j \in \{1, \ldots, k\}: i_{j} \in \{1, \ldots, r\}\}
	\]
	for \(k \geq 2\).
Then \(\CC\) is a vector space basis of \(\mrc\).
\end{theorem}
We refer to such a basis as a \emph{Chen basis}.

\begin{prop}[{\cite[Corollary~1]{Chen51}, \cf \cref{theo:DimensionHallBasis}}]
	Let \(r, c \geq 2\) be integers and let \(\CC\) be a Chen basis of \(\mrc\).
For each \(k \in \{2, \ldots, c\}\), the cardinality of \(\CC_{k}\) is equal to \(M(r, k) := (k - 1) \binom{k + r - 2}{k}\).
\end{prop}

A Chen basis can be used to prove the following:
\begin{lemma}[{\Cf \cref{lem:eigenvaluesphii}}]
	Let \(\phi \in \End(M_{r, c})\) and let, for \(i \in \{1, \ldots, c\}\), \(\phi_{i}\) denote the induced endomorphism on \(\gamma_{i}(M_{r, c}) / \gamma_{i + 1}(M_{r, c})\).
Let \(\lambda = (\lambda_{1}, \ldots, \lambda_{r})\) be the eigenvalues of \(\phi_{1}\), where each eigenvalue is listed as many times as its multiplicity.
Then, for each \(k \in \{1, \ldots, c\}\), the eigenvalues of \(\phi_{k}\) are given by
	\begin{equation}	\label{eq:eigenvaluesphikMetabelian}
		\{\lambda_{i_{1}}\lambda_{i_{2}}\cdots \lambda_{i_{k}} \mid i_{1} > i_{2} \leq i_{3} \leq \ldots \leq i_{k}, \forall j \in \{1, \ldots, k\}: i_{j} \in \{1, \ldots, r\}\}.
	\end{equation}
	In this way, each eigenvalue is also listed as many times as its multiplicity.
\end{lemma}

\begin{theorem}[\Cf \cref{theo:ReidemeisterNumberSymmetricPolynomialEigenvalues}]\label{theo:ReidemeisterNumberSymmetricPolynomialEigenvaluesMetabelian}
	There exist symmetric polynomials \(f_{r, k}^{(2)} \in \sympolnZ\) (with \(1 \leq k \leq c\)) such that, for any \(\phi \in \End(M_{r, c})\), we have
	\[
		R(\phi) = \inftynormb{\prod_{k = 1}^{c} f_{r, k}^{(2)}(\lambda_{1}, \ldots, \lambda_{r})},
	\]
	where \(\lambda_{1}, \ldots, \lambda_{r}\) are the eigenvalues of \(\phi_{1}\).
\end{theorem}

\begin{theorem}[\Cf \cref{theo:ReidemeisterNumbersInTermsOfCoefficientCharPol}]		\label{theo:ReidemeisterNumbersInTermsOfCoefficientCharPolMetabelian}
	There exist polynomials \(\tilde{g}_{r, k}^{(2)} \in \Z[x_{1}, \ldots, x_{r}]\) for \(1 \leq k \leq c\) such that, for any \(\phi \in \End(M_{r, c})\), if \(p_{\phi} = \sum\limits_{i = 0}^{r} a_{i}x^{i}\) is the characteristic polynomial of \(\phi_{1} \in \End(M_{r, c} / \gamma_{2}(M_{r, c}))\) with \(a_{r} = 1\), then
	\[
		R(\phi) = \inftynormb{\prod_{k = 1}^{c} \tilde{g}_{r, k}^{(2)}(a_{0}, \ldots, a_{r - 1})}.
	\]
\end{theorem}

Analogously to \(\SpecR(N_{r, c})\), we can fully determine \(\SpecR(M_{r, c})\) by evaluating the expression above in \(a_{i} \in \Z\) for \(i \in \{1, \ldots, r - 1\}\) and \(a_{0} \in \{\pm 1\}\).

Next, we construct \(\CCinf{}\) in a similar fashion as we constructed \(\HHinf{}\).
Let \(\{X_{i}\}_{i \geq 1}\) be an infinite family.
Let \(\CCinf{} := \bigcup_{k = 1}^{\infty} \CCinf{k}\) where
\[
	\CCinf{1} := \{X_{1}, \ldots, X_{r}, \ldots \}
\]
and
\[
		\CC_{k} := \{[X_{i_{1}}, X_{i_{2}}, \ldots, X_{i_{k}}] \mid i_{1} > i_{2} \leq i_{3} \leq \ldots \leq i_{k}, \forall j \in \{1, \ldots, k\}: i_{j} \geq 1\}
\]
for \(k \geq 2\).
We then also define \(\eta: \CCinf{} \to \M\) similarly as before, where \(\M\) is the set of all monomials on the variables \(\{x_{i}\}_{i \geq 1}\): We set \(\eta(X_{i}) := x_{i}\) for \(i \geq 1\) and
\[
	\eta([X_{i_{1}}, X_{i_{2}}, \ldots, X_{i_{k}}]) := x_{i_{1}} \cdots x_{i_{k}}
\]
for all \(k \geq 2\) and \([X_{i_{1}}, X_{i_{2}}, \ldots, X_{i_{k}}] \in \CCinf{k}\).

\begin{theorem}[\Cf \cref{theo:LinkFkandfrk}]
	Let \(k \geq 1\) be arbitrary.
Define
	\[
		F_{k}^{(2)} := \sum_{X \in \CCinf{k}} \eta(X).
	\]
	and
	\[
		f_{k}^{(2)} := \sum_{i = 0}^{\infty} e_{i}[F^{(2)}_{k}]
	\]
	Then, for all integers \(r \geq 2\)
	\[
		\Phi_{r}\left(f_{k}^{(2)}\right) = f_{r, k}^{(2)}.
	\]
\end{theorem}

\begin{theorem}
	For each \(n \geq 2\), \(F_{n}^{(2)} = s_{n - 1, 1}\).
\end{theorem}
\begin{proof}	
	We need yet another family of symmetric functions, namely the \emph{complete homogeneous symmetric functions}.
Given \(n \geq 1\), the complete homogeneous symmetric function \(h_{n}\) is defined as
	\[
		h_{n} := \sum_{i_{1} \leq i_{2} \leq i_{3} \leq \ldots \leq i_{n}} x_{i_{1}} \cdots x_{i_{n}}.
	\]
	
	By definition of \(\CC_{n}\),
	\begin{align*}
		F_{n}^{(2)}	&=	\sum_{i_{1} > i_{2} \leq i_{3} \leq \ldots \leq i_{n}} x_{i_{1}} \cdots x_{i_{n}}	\\	
					&=	\sum_{i_{1} \leq i_{2} \leq i_{3} \leq \ldots \leq i_{n}} x_{i_{1}} \cdots x_{i_{n}}	 - \sum_{\substack{i_{1} \geq 1 \\ i_{2} \leq i_{3} \leq \ldots \leq i_{n}}	} x_{i_{1}} \cdots x_{i_{n}} \\
					&=	h_{n} - h_{1}h_{n - 1}	\\
					&=	s_{n - 1, 1}
	\end{align*}
	by the Jacobi-Trudi identity \cite[7.16.1~Theorem]{Stanley12a}.
\end{proof}

Note that \(F_{1}^{(2)} = s_{1}\).
For ease of notation, we put \(s_{0, 1} := s_{1}\).

Thus, to compute Reidemeister numbers on free \(c\)-step nilpotent and metabelian groups we have to determine the Schur expansion and/or the expansion in elementary symmetric functions of
\[
	f_{k}^{(2)} = \sum_{i = 0}^{\infty} (-1)^{i} s_{1^{i}}[s_{k - 1, 1}].
\]
The corresponding objectives therefore are
\begin{objective}[\Cf \cref{obj:GroupTheoreticalObjectiveNrc}]
	For all \(i, k \geq 1\), determine the expansion of \(s_{1^{i}}[s_{k - 1, 1}]\) in terms of elementary symmetric functions.
\end{objective}

\begin{objective}[\Cf \cref{obj:DiophantineObjectiveNrc}]
	Suppose we know the expansion of \(s_{1^{i}}[s_{k - 1, 1}]\) in terms of elementary symmetric functions for all \(i, k \geq 1\), say, \(s_{1^{i}}[s_{k - 1, 1}] = \sum_{\lambda \in \Par(ik)} c_{\lambda, i, k}^{(2)} e_{\lambda}\) for some \(c_{\lambda, i, k}^{(2)} \in \Z\).
Given integers \(r, c \geq 2\) and a positive integer \(n \geq 1\), determine whether the equation
	\[
		\pm n = \prod_{k = 1}^{c} \sum_{i = 1}^{M(r, k)} (-1)^{i} \sum_{\substack{\lambda \in \Par(ik) \\ \lambda_{1} \leq r}} c_{\lambda, i, k}^{(2)} y_{\lambda}
	\]
	has a solution with \(y_{1}, \ldots, y_{r - 1} \in \Z\) and \(y_{r} \in \{-1, 1\}\).
\end{objective}

\begin{objective}[\Cf \cref{obj:CombinatorialObjectiveNrc}]	\label{obj:CombinatorialObjectiveMrc}
	For all \(i, k \geq 1\), determine the expansion of \(s_{1^{i}}[s_{k - 1, 1}]\) in terms of Schur functions and find a combinatorial interpretation of the coefficients.
\end{objective}

So, we see that determining the Reidemeister spectrum of free nilpotent and metabelian groups is truly related to determining the Schur expansion of the plethysm of Schur functions.

Since nilpotent groups of nilpotency class at most \(3\) are metabelian, \(N_{r, c}\) and \(M_{r, c}\) are equal for \(c \leq 3\).
Therefore, the objectives for \(N_{r, c}\) and \(M_{r, c}\) coincide in that case.
For higher values of \(c\), information about \(\SpecR(M_{r, c})\) can also yield information about \(\SpecR(N_{r, c})\).
For instance, we have the exact sequence
\[
	1 \to N_{r, c}^{(2)} \to N_{r, c} \to M_{r, c} \to 1.
\]
For \(\phi \in \End(N_{r, c})\), let \(\bar{\phi}\) denote the induced endomorphism on \(M_{r, c}\).
A classical result for Reidemeister numbers states that \(R(\phi) \geq R(\bar{\phi})\) (see \cite[Theorem~1.8]{Heath85}).
Hence, if \(1 \in \SpecR(N_{r, c})\), then also \(1 \in \SpecR(M_{r, c})\).
Put differently, if \(\SpecR(M_{r, c})\) does not contain \(1\), then neither does \(\SpecR(N_{r, c})\); in particular, \(\SpecR(N_{r, c})\) cannot be full.

For some of the plethysms \(s_{1^{i}}[s_{k - 1, 1}]\), the Schur expansion is known.
As before, \(s_{1}[s_{k - 1, 1}] = s_{k - 1, 1}\) for all \(k \geq 1\).
J.\;Carbonara, J.\;Remmel and M.\;Yang explicitly compute all plethysms of the form \(s_{1, 1}[s_{a, 1^{b}}]\), where \(a, b \geq 1\) \cite[Theorem~3]{CarbonaraRemmelYang92}.
Concretely, put \(\lambda = (1, 1)\), \(\mu = (k - 1, 1)\) with \(k \geq 2\) and \(s_{1, 1}[s_{\mu}] = \sum_{\nu} a_{\lambda, \mu}^{\nu} s_{\nu}\).
Their results then state the following:
\begin{itemize}
	\item If \(\nu = (2k - m, 1^{m})\) for some \(m \in \{1, \ldots, 2k - 1\}\), then
	\[
		a_{\lambda, \mu}^{\nu} = \begin{cases}
			1	&	\mbox{if } m = 2,	\\
			0	&	\mbox{otherwise.}
		\end{cases}
	\]
	\item If \(\nu = (q, p, 2^{l}, 1^{m})\) for some integers \(q \geq p \geq 2\) and \(l, m\) with \(q + p + 2l + m = 2k\), then
	\[
		a_{\lambda, \mu}^{\nu} = \begin{cases}
			1	&	\mbox{if \(q + p \in \{2k - 2, 2k\}\) and \(\frac{m}{2} + p\) is even,}	\\
			1	&	\mbox{if } q + p = 2k - 1,	\\
			0	&	\mbox{otherwise.}
		\end{cases}
	\]
	\item If \(\nu\) is of a different form, then \(a_{\lambda, \mu}^{\nu} = 0\).
\end{itemize}

On the other hand, C.\;Carr\'{e} and B.\;Leclerc give a combinatorial interpretation of the coefficients in the Schur expansion of \(s_{1, 1}[s_{\lambda}]\), where \(\lambda\) is any partition \cite[Corollary~5.5]{CarreLeclerc95}.
Besides these explicit plethysms, there has been done some research into plethysms of so-called \emph{hook shaped} partitions \cite{LangleyRemmel04}, albeit with only partial Schur expansions.

\section*{Acknowledgments}
The author thanks Lore De Weerdt for her useful remarks and suggestions.

\printbibliography

\end{document}